\documentclass[11pt]{amsart}%{article}

\usepackage{xcolor}

\usepackage{amsfonts}
\usepackage{amsmath}
\usepackage{amssymb}

\newcommand{\R}{\mathbb R}
\newcommand{\N}{\mathbb N}

\newcommand{\HH}{\mathcal{H}}

\newcommand{\Rn}{\mathbb R^{n}}

\newtheorem{theorem}{Theorem}[section]
\newtheorem{lemma}[theorem]{Lemma}
\newtheorem{prop}[theorem]{Proposition}
\newtheorem{coro}[theorem]{Corollary}
\newtheorem{claim}[theorem]{Claim}
\newtheorem{definition}[theorem]{Definition}
\newtheorem{example}[theorem]{Example}

\title{Centro-sectional measures for log-concave functions}
\author{K\'aroly J. B\"or\"oczky}
\address{Alfr\'ed R\'enyi Institute of Mathematics, Hungarian Academy
  of Sciences, Re\'altanoda u. 13-15, H-1053 Budapest, Hungary, and
Institute of Mathematics, E\"otv\"os University, P\'azm\'any P\'eter s\'et\'any 1/c, H-1117, Budapest, Hungary}
\email{boroczky.karoly.j@renyi.hu}

\author{Jinrong Hu}
\address{Institut f\"{u}r Diskrete Mathematik und Geometrie, Technische Universit\"{a}t Wien, Wiedner Hauptstrasse 8-10, 1040 Wien, Austria}
\email{jinrong.hu@tuwien.ac.at}

\author{Jiaqian Liu}
\address{School of Mathematics and statistics, Henan University, Jinming Avenue, 475001, Kaifeng, China}
\email{liujiaqian@henu.edu.cn}

\thanks{2020 \emph{Mathematics Subject Classification}: 52A40 (52A38)\\
\emph{Keywords}: Log-concave functions; Centro-sectional measures}

\begin{document}

\begin{abstract}
We introduce centro-sectional measures with parameters $q,m$ for log-concave functions on $\mathbb{R}^n$, defined in terms of the $q$-th moments of their Radon transforms with respect to the Haar measure on $m$-dimensional subspaces, where $m=1,\ldots,n-1$, and establish the corresponding variational formulas. Our measures generalize the notion of dual curvature measure if $q=1$, and are related to the Sine transform if $q=2$. In line with the coarea formula for log-concave functions as BV functions, the variational formulas give rise to the Euclidean centro-sectional measures and the spherical centro-sectional measures. In the symmetric setting, we solve the associated even functional centro-sectional Minkowski problem, which asks which pairs of measures can arise as the centro-sectional measures of an even log-concave function.

\end{abstract}
\maketitle

\section{Introduction}

A recurring theme in convex geometry is to recover a convex body from geometric information associated with it. Such reconstruction problems, commonly referred to as Minkowski-type problems, occupy a central position in the field and have stimulated extensive developments over the past century. 
The geometric data appearing in Minkowski-type problems are typically geometric measures obtained from natural geometric invariants. In the classical Brunn--Minkowski theory, the fundamental invariants are the quermassintegrals, whose variational formulas give rise to the area measures of Aleksandrov, Fenchel, and Jessen and the curvature measures introduced by Federer. 
The prototype of these reconstruction problems is the classical Minkowski problem for the surface area measure, solved by Minkowski \cite{M897,M903} in the discrete and absolute continuous case around 1900, and later by  Aleksandrov  \cite{A39,A42} in general. Regularity of the solution has been intensively investigated in the second half of the 20th century; see, for example, \cite{Caf90a,Caf90b,CY76,FJ38,N53,P52}. 

A dual picture emerged with Lutwak's introduction of the dual Brunn--Minkowski theory in the 1970s \cite{L75}. Recently, for $q\in\R$, Huang, Lutwak, Yang, and Zhang \cite{HLYZ16} introduced the dual curvature measure $\widetilde{C}_q(K,\cdot)$ through the variational theory of the dual intrinsic volumes $\widetilde{V}_q(K)$ (equivalently, the corresponding dual quermassintegrals) of a convex body $K\subset\Rn$ with $o\in{\rm int}\,K$. 
These measures naturally extend the role played by Federer's curvature measures in the classical theory and lead to the formulation of the dual Minkowski problem by \cite{HLYZ16}, and their work has initiated a rapidly growing research direction; see \cite{BF19,BLYZ19,BLYZ20,CL18,HJ19,HZ18,LSW20,Z17,Z18}. We note that for $q\neq0$ and an $o$-symmetric convex body $K\subset\Rn$, the dual intrinsic volume is 
$$
\widetilde{V}_q(K)=\frac{\omega_n}{2^q}\int_{{\rm G}(n,1)}{
\rm V}_1(K\cap\xi)^q\,d\nu_1(\xi),
$$
where  ${\rm G}(n,m)$ is the Grassmannian of linear $m$-planes of $\Rn$ equipped with  the Haar probability measure $\nu_{m}$, $m=1,\ldots,n-1$, ${
\rm V}_m$ stands for the $m$-dimensional volume, and $\omega_n={
\rm V}_n(B^n)$ for the Euclidean centered unit ball $B^n\subset\R^n$. A key property of the  celebrated dual curvature measure $\widetilde{C}_q(K,\cdot)$ on $S^{n-1}$ introduced by \cite{HLYZ16} is the variational formula that if $q\neq 0$ and $L\subset \R^n$ is any convex body, then
\begin{equation}
\label{tildeV-tildeC}
\lim_{t\to 0^+}\frac{\widetilde{V}_q(K+tL)-\widetilde{V}_q(K)}t=q\int_{S^{n-1}}\frac{h_L}{h_K}\,d\widetilde{C}_q(K,\cdot).
\end{equation}

A similar notion originating from integral geometry is the ${\rm SL}(n)$ invariant dual affine quermassintegrals $\widetilde{\Phi}_{n-m}(K)$ of $K$, $m=1,\ldots,n-1$, that were proposed by Lutwak in the 1980s (see \cite[P. 515]{S14}) and whose main properties were established by Grinberg \cite{Gri91},
Busemann, Straus \cite{BuS60} and Gardner \cite{Gar07} (see also
Milman, Yehudayoff \cite{MiY23}), and are defined as 
$$
\widetilde{\Phi}_{n-m}(K)= \frac{\omega_n}{\omega_m}\left(\int_{{\rm G}(n,m)}{\rm V}_m(\xi\cap K)^n\,d\nu_{m}(\xi)\right)^{\frac1n}.
$$

A common generalization of these notions have been recently proposed by Cai, Leng, Wu, and Xi \cite{CLWX26}, considering a two-parameter family of centro-sectional functionals $\widetilde{\Psi}_{m,q}(K)$, $q\in\R$, $m=1,\ldots,n-1$, of a convex body $K\subset\R^n$ with $o\in{\rm int}\,K$; namely,
\begin{align*}
\widetilde{\Psi}_{m,q}(K)=&\int_{{\rm G}(n,m)}{\rm V}_m(\xi\cap K)^q\,d\nu_{m}(\xi),&&q\neq 0,\\
\widetilde{\Psi}_{m,0}(K)=&\int_{{\rm G}(n,m)}\log ({\rm V}_m(\xi\cap K))\,d\nu_{m}(\xi),&&q=0.
\end{align*}
 This notion includes dual quermassintegral and dual affine quermassintegral as special cases, namely, when $q=1$,  then $\widetilde{\Psi}_{m,1}(K)=\frac{\omega_m}{\omega_n}\widetilde{V}_{m}(K)$; and when $q=n$, then $\widetilde{\Psi}_{m,n}(K)=(\omega_m/\omega_n)^n\widetilde{\Phi}_{n-m}(K)^{n}$.
The paper \cite{CLWX26} established the property that  the differentials of $\widetilde{\Psi}_{m,q}(\cdot)$ lead to the variational centro-sectional measures $\widetilde{A}_{m,q}(K,\cdot)$ on $S^{n-1}$ analogously to \eqref{tildeV-tildeC}. 

The centro-sectional measures provide a unified framework encompassing dual curvature measures and their affine analogues. More precisely, when $q=1$, then $\widetilde{A}_{m,1}(K,\cdot)=\frac{\omega_m}{\omega_n}\widetilde{C}_{m}(K,\cdot)$; and when $q=n$, then $\widetilde{A}_{m,n}(K,\omega)$, is the affine dual curvature measure based on the notion of dual affine quermassintegral $\widetilde{\Phi}_{n-m}(K)$ (see \cite{CLWX25}).
The even Minkowski problem associated with centro-sectional measures was also solved in \cite{CLWX26}. More recently,  \cite{BLY26+} extended these results to the $L_p$ setting.

The purpose of the present paper is to extend this centro-sectional theory from convex bodies to log-concave functions. Such a passage is natural in functional convex geometry, where convex bodies are embedded into the class of log-concave functions through their characteristic functions, see for example Colesanti, Fragal\`a \cite{CF13}, Cordero-Erausquin, Klartag \cite{CK15} and Falah, Rotem \cite{FR26} in the case of the functional Minkowski problem.  Recently, Xi, Zhao \cite{XZ26} introduced fractional integral affine surface areas and the associated fractional affine area measures through a  variational framework, and studied the corresponding Minkowski problems. While their construction is based on directional chord data, rather than centro-section, it illustrates a closely related principle: nonlocal geometric measures can be obtained as first variations of integral-geometric functionals. This perspective is also central to the functional centro-sectional theory developed here.
Our resulting theory generalizes the 
theory of dual curvature measures for log-concave functions introduced by Huang, Liu, Xi, and Zhao \cite{HLXZ24}; see also \cite{FXY22,FYZZ25,R23,U25}. We recall  that a function $f:\Rn\to[0,\infty)$ is log-concave if $f((1-\lambda)x+\lambda y)\geq f(x)^{1-\lambda}f(y)^\lambda$ for any $x,y\in\Rn$ and $\lambda\in[0,1]$; or equivalently, $f=e^{-\varphi}$ for a convex function $\varphi:\Rn\to(-\infty,\infty]$. The domain of such $f$ is the convex set $D_f=\{f>0\}$ which has non-empty interior if $\int_{\R^n} f>0$. 

Let $m\in\{1,\ldots,n-1\}$. For a log-concave $f$ with $0<\int_{\Rn} f<\infty$ and $o\in{\rm int}\,D_f$,  the $m$-dimensional Radon transform ${\rm R}_m f\in L_\infty({\rm G}(n,m))$  is defined by
 \begin{equation*}\label{Rm}
{\rm R}_m f(\xi)=\int_{\xi}f(x)d\HH^m(x)=\|f|_{\xi}\|_1
 \end{equation*} 
for $\xi\in {\rm G}(n,m)$ (see \eqref{log-concave-f-xi-A0A1} for finiteness).  We note that in this case, ${\rm R}_m f$ is a continuous function of $\xi\in {\rm G}(n,m)$ (see  Lemma~\ref{fxi-representations}). We also mention that when $m=n$, ${\rm R}_n f$ coincides with the functional volume of $f$; see Falah-Rotem \cite{FR26}.
Tziotziou \cite{T26} extended the dual affine intrinsic volume $\widetilde{\Psi}_{m,n}(K)$ of a convex body $K$ to a log-concave function $f$ on $\Rn$ as
$$
\widetilde{\Psi}_{m,n}(f)=\int_{{\rm G}(n,m)}({\rm R}_mf(\xi))^nd\nu_{m}(\xi).
$$
For $q\in\R$, motivated by the work of Cai, Leng, Wu, Xi \cite{CLWX26}, we consider the $q$-th dual centro-sectional intrinsic volume
\begin{align*}
\widetilde{\Psi}_{m,q}(f)=&\int_{{\rm G}(n,m)}({\rm R}_mf(\xi))^q\,d\nu_{m}(\xi),&& q\neq 0,\\
\widetilde{\Psi}_{m,0}(f)=&\int_{{\rm G}(n,m)}\log ({\rm R}_mf(\xi))\,d\nu_{m}(\xi),&&q=0.
\end{align*}
In particular, if $f=\mathbf{1}_K$ for a convex body $K\subset\R^n$ with $o\in{\rm int}\,K$, then $\widetilde{\Psi}_{m,q}(f)=\widetilde{\Psi}_{m,q}(K)$.

The sum operation on the space of log-concave functions is sup-convolution. For log-concave functions $f, g$ on $\Rn$ and $t>0$, the sup-convolution $f\oplus t\cdot g$ is the upper semicontinuous log-concave function
\begin{equation*}
	 ( f\oplus t\cdot g )  (z)= \sup_{x+t y=z} f(x)g(y)^t.
\end{equation*}
We recall (see Section~\ref{secConvexFunctions}) that for any  function $\psi:\Rn\to(-\infty,\infty]$, its Legendre transform $\psi^*:\Rn\to(-\infty,\infty]$ is defined by the formula
\begin{equation*}
 \label{Legendre-def}
\psi^*(x)=\sup_{y\in\Rn}x\cdot y-\psi(y), 
\end{equation*}
that is a lower semicontinuous convex function.  If $\psi$ itself is a lower continuous convex function, then $(\psi^*)^*=\psi$.
We note that if $f=e^{-\varphi}$ and $g=e^{-\psi}$ are upper semicontinuous log-concave functions on $\R^n$, then $f\oplus t\cdot g=e^{-(\varphi^*+t\psi^*)^*}$.

In the case of Minkowski-type problems for log-concave functions, two measures are associated to a log-concave function $f$, one on $\R^n$ and one on $S^{n-1}$ (see, for example, Colesanti, Fragal\`a \cite{CF13}, Falah, Rotem \cite{FR26} or Huang, Liu, Xi, Zhao \cite{HLXZ24}). The main explanation is the coarea formula for functions of bounded variations as it is discussed in Section~\ref{secBasic-Estimates} (see  
\eqref{coarea-BV-eq} for the special case we need).
Motivated by the Brunn-Minkowski theory of convex bodies,
it is natural to consider the problem when the limit
\begin{equation}\label{eq limit}
\delta_{m,q}(f,g)=\lim_{t\rightarrow 0^+}\frac{\widetilde{\Psi}_{m,q}(f\oplus t\cdot g)-\widetilde{\Psi}_{m,q}(f)}{t}
\end{equation}
exists, and to find a representation of it as an integral with respect to a variational measure based on $\widetilde{\Psi}_{m,q}(f)$. Our first result Theorem~\ref{Amq-pair-measures} solves this problem.

\begin{theorem}
\label{Amq-pair-measures}
Let $q\in \R$,  and $m\in\{1,\ldots,n-1\}$. Let $f=e^{-\varphi}$ be an  upper semicontinuous  log-concave function  on $\Rn$ such that $0<\int_{\Rn} f<\infty$, $\sup f=f(o)$ and $o\in{\rm int}\,D_f$. There exist finite Borel measures $\widetilde{A}_{m,q}^{e}(f,\cdot)$ on $\R^n$ and $\widetilde{A}_{m,q}^{s}(f,\cdot)$ on $S^{n-1}$ such that for any compactly supported  upper semicontinuous log-concave function   $g=e^{-\psi}$ with $g(o)>0$, if $q\neq 0$, then
 \begin{align}
 \label{Amq-pair-measures-eq}
 \delta_{m,q}(f,g) = &q\int_{\Rn} \psi^{*}(x)d\widetilde{A}^{e}_{m,q}(f,x)+ q\int_{S^{n-1}} h_{D_g}(u)d\widetilde{A}^{s}_{m,q}(f,u),\\
 \label{Amq-pair-measures-eq2}
 \delta_{m,0}(f,g) = &\int_{\Rn} \psi^{*}(x)d\widetilde{A}^{e}_{m,0}(f,x)+ \int_{S^{n-1}} h_{D_g}(u)d\widetilde{A}^{s}_{m,0}(f,u).
 \end{align}
 \end{theorem}

Note that $h_{D_g}$ is the so-called recession function of $\psi^*$ in Theorem~\ref{Amq-pair-measures} (cf. \eqref{barpsi-dompsi-bounded} in Section~\ref{secConvexFunctions}), and the support of $g$ might be lower dimensional in Theorem~\ref{Amq-pair-measures}.  

Let us describe the centro-sectional measures $\widetilde{A}_{m,q}^{e}(f,\cdot)$ on $\R^n$ and $\widetilde{A}_{m,q}^{s}(f,\cdot)$ on $S^{n-1}$ of a log-concave function occurring in Theorem~\ref{Amq-pair-measures}, where the definition makes sense even if the maximum is not attained at the origin (see Section~\ref{secCoarea} for details).  
We note that given a closed convex set $K\subset\R^n$ with ${\rm int}\,K\neq \emptyset$, 
 there exists a unique exterior normal $u_K(x)\in S^{n-1}$ to $K$ at $x$ for $\HH^{n-1}$ a.e. $x\in\partial K$, and
the $m$-dimensional spherical dual Radon transform $\widetilde{\mathcal{R}}^*_m$ occurring in Definition~\ref{tildeAdef} is defined in Section~\ref{secPreliminaries}.

\begin{definition} 
\label{tildeAdef}
Let $q \in \R$ and $m\in\{1,\ldots,n-1\}$. Let $f=e^{-\varphi}$ be an upper semicontinuous log-concave function  on $\Rn$ such that $0<\int_{\Rn} f<\infty$ and  $o\in{\rm int}\,D_f$. For the associated centro-sectional variational Borel measures $\widetilde{A}_{m,q}^{e}(f,\cdot)$ on $\R^n$ and $\widetilde{A}_{m,q}^{s}(f,\cdot)$ on $S^{n-1}$, if $\Omega\subset\R^n$ and $\omega\subset S^{n-1}$ are Borel sets, then
\begin{align*}
\widetilde{A}_{m,q}^{e}(f,\Omega)
:=&\int_{\{x\neq o:\nabla\varphi(x)\in\Omega\}} \|x\|^{m-n}f(x) \left(\widetilde{\mathcal{R}}^*_m({\rm R}_mf)^{q-1}\right)\left(\frac{x}{\|x\|}\right)\,dx,\\
\widetilde{A}_{m,q}^{s}(f,\omega):=
&\int_{\{x\in \partial'D_f:u_{D_f}(x)\in\omega\}}\|x\|^{m-n}f(x) \left(\widetilde{\mathcal{R}}^*_m({\rm R}_mf)^{q-1}\right)\left(\frac{x}{\|x\|}\right)\,dx.
\end{align*}
\end{definition}

\noindent\textbf{The functional centro-sectional Minkowski problem.} Let $q\in \R$  and $m\in\{1,\ldots,n-1\}$. For non-trivial Borel measures $\mu$  on $\R^n$  and $\nu$ on $S^{n-1}$, find the necessary and sufficient conditions on $\mu$ and $\nu$ so that there exists  an upper semicontinuous log-concave function $f$ on $\Rn$ with $0<\int_{\Rn} f<\infty$ and $o\in{\rm int}\,D_f$ such that 
\begin{equation*}
\label{Func-cen-MP}
\mu=\widetilde{A}_{m,q}^{e}(f,\cdot) \mbox{ \ and \ } \nu=\widetilde{A}_{m,q}^{s}(f,\cdot).
\end{equation*}

If the measure $\mu$ has a density $g\in L_1(\Rn)$, then the equation $ \mu=\widetilde{A}_{m,q}^{e}(f,\cdot)$ is equivalent  to the Monge-Amp\`ere equation
\begin{align}
\label{Amqp-Monge-Ampere}
 \left(\widetilde{\mathcal{R}}^*_m\left({\rm R}_me^{-\varphi(x)}\right)^{q-1}\right)\left(\frac{x}{\|x\|}\right)||x||^{m-n}e^{-\varphi(x)}={\rm det} (\nabla^2 \varphi(x))\cdot g(\nabla\varphi(x)).
\end{align}

It is worth mentioning that for $q=2$, $m\in\{2,\ldots,n-1\}$,  by a direct computation, the nonlocal term in the functional
centro-sectional Monge-Amp\`ere equation is given by the
spherical $(m-n)$-Sine transform (see \cite{Rub13}) of the $m$-th power of
the radial function of the associated Keith Ball body $K_m(f)$ (cf. \eqref{Ball-body}), thus \eqref{Amqp-Monge-Ampere} reduces to 
\begin{equation*}
\frac{c_{n,m}e^{-\varphi(o)}}{m}\,
\left[
\mathcal{S}_{m-n}
\left(\varrho_{K_m(e^{-\varphi})}^m\right)
\right]
\left(\frac{x}{\|x\|}\right)
\|x\|^{m-n}e^{-\varphi(x)}
=
\det\bigl(\nabla^2\varphi(x)\bigr)\cdot 
g\bigl(\nabla\varphi(x)\bigr),
\label{eq:q2-functional-centro-sine}
\end{equation*}
where 
\[
c_{n,m}
=
\frac{\Gamma(n-1)}
     {(2\pi)^{\,n-m}\Gamma(m-1)},
\]
and 
$\mathcal{S}_{m-n}$ is the  spherical $(m-n)$-Sine transform defined for $h\in L_\infty(S^{n-1})$ by
\[
(\mathcal{S}_{m-n}h)(u)
=
\int_{S^{n-1}}
|\sin\angle(u,v)|^{m-n}
h(v)\,d\mathcal{H}^{n-1}(v), \qquad u\in S^{n-1}.
\]

A Borel measure $\mu$ on $\R^n$ is said to have finite first moment if $\int_{\R^n}|x|d\mu(x)<\infty$. Our other main result is the characterization of even functional centro-sectional measures.

\begin{theorem}
\label{Amq-even-Minkowski-pair-measures}
Let $q\in\R$ and $m\in\{1,\ldots,n-1\}$. Let $\mu$ and $\nu$ be even Borel measures on $\R^n$ and on $\mathbb{S}^{n-1}$, respectively.
There exists an even upper semicontinuous log-concave function $f$ on $\Rn$ with $0<\int_{\Rn} f<\infty$ such that
\begin{align*}
\mu=&\widetilde{A}_{m,q}^{e}(f,\cdot),\\
\nu= &\widetilde{A}_{m,q}^{s}(f,\cdot), 
\end{align*}
 if and only if the pair $(\mu,\nu)$ satisfies the following conditions:
 \begin{itemize}
\item $0<\mu(\mathbb{R}^n)<\infty $ if $q\neq 0$, and  $\mu(\mathbb{R}^n)=1$ if $q=0$.
 
 \item $\nu$ is finite and $\mu$ has finite first moment.
 
 \item The union of the supports of $\mu$ and $\nu$ is not contained in any hyperplane.
 \end{itemize}
% \begin{itemize}
% \item either $q\neq 0$ and  $0<\mu(\R^n)<\infty$,
% \item or $q=0$ and $\mu(\R^n)=1$.
% \end{itemize}
\end{theorem}

Let $f=e^{-\varphi}$ be an upper semicontinuous log-concave function  on $\Rn$ such that $0<\int_{\Rn} f<\infty$ and  $o\in{\rm int}\,D_f$. For $q>0$, Huang, Liu, Xi, Zhao \cite{HLXZ24} introduced the functional dual intrinsic volume
$$
\widetilde{V}_q(f)=\int_{\R^n}\|x\|^{q-n}f(x)\,dx
$$
that is always finite. We note that  
 the ``classical" dual intrinsic volume $\widetilde{V}_q(K)$ of a convex body $K\subset\R^n$ with $o\in{\rm int}\,K$ whose variational so-called dual curvature measures were introduced by Huang, Lutwak, Yang, Zhang \cite{HLYZ16}
 can be represented as $\widetilde{V}_q(K)=\frac{q}{n}
\widetilde{V}_q\left(\mathbf{1}_K\right)$.
The paper \cite{HLXZ24} also defined the associated functional dual curvature measures $\widetilde{C}_{q}^{e}(f,\cdot)$ on $\R^n$ and $\widetilde{C}_{q}^{s}(f,\cdot)$ on $S^{n-1}$ in a way such that if $\Omega\subset\R^n$ and $\omega\subset S^{n-1}$ are Borel sets, then
\begin{align*}
%\label{tildeCe-def}
\widetilde{C}_{q}^e(f,\Omega)
:=&\int_{\{x\neq 0:\nabla\varphi(x)\in\Omega\}} \|x\|^{q-n}f(x) \,dx,\\
%\label{tildeCs-def}
\widetilde{C}_{q}^s(f,\omega):=
&\int_{\{x\in \partial'D_f:u_{D_f}(x)\in\omega\}}\|x\|^{q-n}f(x) \,d\HH^{n-1}(x).
\end{align*}
We note that if $m=1,\ldots,n-1$, then
$$
\widetilde{\Psi}_{m,1}(f)=\frac{\omega_m}{\omega_n}\cdot \widetilde{V}_m(f),
$$
$\widetilde{A}^e_{m,1}(f)=\frac{\omega_m}{\omega_n}\cdot \widetilde{C}^e_m(f)$ and $\widetilde{A}^s_{m,1}(f)=\frac{\omega_m}{\omega_n}\cdot \widetilde{C}^s_m(f)$,
therefore, our Theorems~\ref{Amq-pair-measures} and \ref{Amq-even-Minkowski-pair-measures}
improve the corresponding results in \cite{HLXZ24} for $q=1,\ldots,n-1$ as \cite{HLXZ24} needed the following regularity assumptions on $f$ besides $f(o)=\max f$:
There exists $\alpha\in(0,1)$ such that
\begin{equation*}
\label{Huang-Liu-regularity}
\lim_{x\to o}\frac{f(o)-f(x)}{\|x\|^{\alpha+1}}<\infty.
\end{equation*}
In turn, \cite{HLXZ24} could only solve the even Minkowski problem for $f$ when $f(x)=0$ for $\HH^{n-1}$ for a.e. $x\in\partial\{f>0\}$ (in this case, only the Euclidean measure $\widetilde{C}_q^e(f,\cdot)$ occurs, not a pair of measures). 
We have checked that our method can be extended to verify the analogues of Theorem~\ref{Amq-pair-measures} and 
Theorem~\ref{Amq-even-Minkowski-pair-measures} for $\widetilde{C}_{q}^e(f,\Omega)$ and $\widetilde{C}_{q}^s(f,\Omega)$ for any $q>0$ without the regularity assumptions in \cite{HLXZ24}.

Concerning the structure of the paper,
Section~\ref{secPreliminaries} introduces some basic notions, including the spherical Radon transform and its dual. 
Section~\ref{secConvexFunctions} focuses on the basic notions associated to convex functions like Legendre transform and epiconvergence. The coarea formula holds for upper semicontinuous log-concave functions viewed as BV functions, and in turn, the functional centro-sectional measures are introduced and discussed in Section~\ref{secCoarea}, and some additional estimates are provided in
Section~\ref{secBasic-Estimates}.
An estimate for the variation of the moments over a convex body, which plays a fundamental role in the proof of Theorem~\ref{Amq-pair-measures},
is verified in Section~\ref{secpower-integral}. The necessary two-sided limits at zero for 
Theorem~\ref{Amq-pair-measures} and 
Theorem~\ref{Amq-even-Minkowski-pair-measures} are established in Section~\ref{secVariational-formulas}. Finally, Theorem~\ref{Amq-pair-measures} and 
Theorem~\ref{Amq-even-Minkowski-pair-measures} are proved in
Section~\ref{secTheorem1.1} and
Section~\ref{secTheorem1.3}, respectively.

\section{Preliminaries}
\label{secPreliminaries}

The scalar product of $x,y\in\R^n$
is denoted by $x\cdot y$,  and $\|x\|=\sqrt{x\cdot x}$ stands for the corresponding Euclidean norm. Accordingly, we write $B^n=\{x\in\R^n:\|x\|\leq 1\}$ to denote the Euclidean unit ball, and set $S^{n-1}=\partial B^n$.
We recall that a convex body in $\R^n$ is a compact  convex set with non-empty interior.  The $k$-dimensional Hausdorff measure of a Borel subset $X$ of $\R^n$ is denoted by $\HH^k(X)$, normalized in a way such that $\HH^k(X)=V_k(X)$ is the $k$-dimensional Lebesgue measure if $X\subset\R^k$. We set any measure of the empty set  and any integral over the empty set to be zero. If $f\in L_1(\R^n)$, then we also write  $\int_{\R^n} f$ or $\int_{\R^n} f(x)\,dx$ to denote the Lebesgue integral of $f$. In addition, if $g:\partial K\to\R$ is integrable with respect to $\HH^{n-1}$ for a closed convex set $K\subset\R^n$ with $K\neq\R^n$ and ${\rm int}\,K\neq \emptyset$, then we frequently write $\int_{\partial K}g\,d\HH^{n-1}=\int_{\partial K}g(x)\,dx$.

For any finite dimensional real vector space $V$ and $m=0,\ldots,{\rm dim}\,V$, we write ${\rm G}(V,m)$ to denote the Grassmannian of linear subspace of dimension $m$ of $V$, and $\nu_{V,m}$ to denote the unique Haar probability measure on ${\rm G}(V,m)$. In addition, we typically write ${\rm G}(n,m)$ and $\nu_{m}$ for ${\rm G}(\R^n,m)$ and $\nu_{\R^n,m}$. When integrating a function of $\xi\in{\rm G}(n,m)$ over ${\rm G}(n,m)$, we frequently simply write $d\xi$ instead of $d\nu_{m}(\xi)$. 

For $m\in\{1,\ldots,n-1\}$ and Borel measurable $\eta\in L_\infty(S^{n-1})$, the spherical $m$-dimensional Radon transform $\widetilde{\mathcal{R}}_m\eta\in L_\infty({\rm G}(n,m))$  is defined by the formula
\begin{equation*}
\label{Rm-sphere}
\widetilde{\mathcal{R}}_m \eta (\xi)=\int_{\xi\cap S^{n-1}}\eta(x)d\HH^{m-1}(x)
 \end{equation*}
for $\eta\in L_\infty(S^{n-1})$ and $\xi\in {\rm G}(n,m)$. In turn, if $\Upsilon\in  L_\infty({\rm G}(n,m))$ is Borel measurable, then the spherical dual Radon transform $\widetilde{\mathcal{R}}^*_m\Upsilon\in L_\infty(S^{n-1})$ satisfies
\begin{equation}  
\label{dualRm-sphere}
\widetilde{\mathcal{R}}^*_m \Upsilon (u)=\frac{m\omega_m}{n\omega_n}\int_{\zeta\in {\rm G}(u^\bot,m-1)}\Upsilon(\zeta+\R u)\,d\nu_{u^\bot,m-1}(\zeta)
 \end{equation}
for $u\in S^{n-1}$. Helgason's duality formula says that if $\eta\in L_\infty(S^{n-1})$ and $\Upsilon\in  L_\infty({\rm G}(n,m))$ are Borel measurable, then
\begin{equation}
\label{Rm-dualRm-sphere}
\int_{{\rm G}(n,m)}\Upsilon \cdot \widetilde{\mathcal{R}}_m\eta\,d\nu_{m}=
\int_{S^{n-1}}\eta\cdot \widetilde{\mathcal{R}}^*_m \Upsilon (u)\,d\HH^{n-1}.
 \end{equation}

For an integer $k\geq 1$ and a bounded convex set $F\subset\R^k$ with non-empty interior, we define the ``generalized" radial function $\varrho_F\in L_\infty(S^{k-1})$ by the formula
$$
\varrho_F(u)=
\left\{\begin{array}{rl}
0&\mbox{ if }\R u\cap F=\emptyset,\\
\sup\{r\in\R:ru\in F\}&\mbox{ if }\R u\cap F\neq \emptyset
\end{array}\right.
$$ 
for $u\in S^{k-1}$, where $\varrho_F=\varrho_{{\rm cl}\,F}$ is Borel measurable. In order to be able to handle the empty set, we just set 
$$
\varrho_{\emptyset}(u)=0 \mbox{ \  for }u\in S^{k-1}.
$$
Writing ${\rm sign}\,t=1,0,-1$ if $t>0$, $t=0$ or $t<0$, respectively, we deduce from integration in polar coordinates that if $\theta>0$, then
\begin{align}
\label{polar-volume}
V_k(F)=&\frac1k\int_{S^{k-1}}({\rm sign}\,\varrho_F)|\varrho_F|^kd\HH^{k-1},\\
\label{polar-volume-q}
\int_{F}\|x\|^{\theta-k}\,d\HH^k(x)=&\frac1\theta\int_{S^{k-1}}({\rm sign}\,\varrho_F)|\varrho_F|^\theta d\HH^{k-1}.
 \end{align}
In particular, if $F\subset\R^n$ is a convex body and $m\in\{1,\ldots,n-1\}$, then \eqref{polar-volume-q} and integrating in polar coordinates imply
\begin{align}
\nonumber
\int_{{\rm G}(n,m)}\widetilde{\mathcal{R}}_m(\left({\rm sign}\,\varrho_F\right)|\varrho_F|^m)\,d\nu_{m}=&\int_{S^{n-1}}\left({\rm sign}\,\varrho_F\right)|\varrho_F|^m \widetilde{\mathcal{R}}^*_m\mathbf{1}_{{\rm G}(n,m)}\,d\HH^{n-1}\\
\label{rho-tom-xi-inner-int}
=&\frac{m^2\omega_m}{n\omega_n}\int_F\|x\|^{m-n}\,dx.
\end{align}

If $F,K\subset \R^n$ are compact convex sets, then their Hausdorff distance is
$$
\delta_H(F,K)=\min\{r\geq 0:F\subset K+r B^n\mbox{ and }K\subset F+rB^n\},
$$
that turns the family of compact convex sets in $\R^n$ into a locally compact  metric space. Using this topology, the Blaschke Selection theorem says that if $R>0$ and $F_k\subset RB^n$ are compact convex sets for $k\in\N$, then there exists a subsequence $\{F_{k'}\}$ and a compact convex set $F\subset\R^n$ such that
\begin{equation}
\label{Blaschke-Selection}
\lim_{k'\to\infty}F_{k'}=F.
\end{equation}
 In particular, we have the following statement.

\begin{claim}
\label{convergence-by-radial}
Let $F,F_k\subset\R^n$, $k\in\N$ be convex bodies containing $o$ in their interiors. Then $F_k$ tends to $F$ if and only if for any $u\in S^{n-1}$, we have $\lim_{k\to\infty}\varrho_{F_k}(u)=\varrho_F(u)$. 
\end{claim}

For $t\in\R$, and convex body $F\subset\R^n$ and compact convex set $L\subset\R^n$ with $o\in L$, we consider 
\begin{equation}
\label{FoplustL}
F\oplus t\cdot L=\{x\in\R^n:x\cdot u\leq h_{F}(u)+th_L(u)\mbox{ \ for }u\in S^{n-1}\},
\end{equation}
that is a compact convex set if it is not the empty set. In particular, if $t\geq 0$, then $F\oplus t\cdot L=F+tL$ (the Minkowski sum),  and if $t< 0$, then $F\oplus t\cdot L=\{x\in \R^n: x+|t|L\subset F\}$, which might be the empty set. In addition, if $L\subset L'$ for a compact convex set $L'\subset\R^n$ and $t<t'$, then 
\begin{equation*}
\label{FtLmonotonicity}
F\Delta (F\oplus t\cdot L)\subset F\Delta ( F\oplus t\cdot L')\mbox{ \ and \ }
F\oplus t\cdot L\subset F\oplus t'\cdot L. 
\end{equation*}

% For $m\in\{1,\ldots,n-1\}$ and $q\in\R$, Cai, Leng, Wu, and Xi \cite{CLWX26} defined the centro-sectional functionals $\widetilde{\Psi}_{m,q}(K)$ of a convex body $K\subset\R^n$ with $o\in{\rm int}\,K$ as
% \begin{align*}
% \widetilde{\Psi}_{m,q}(K)=&\int_{{\rm G}(n,m)}{\rm v}_m(\xi\cap K)^q\,d\nu_{m}(\xi),&&q\neq 0,\\
% \widetilde{\Psi}_{m,0}(K)=&\int_{{\rm G}(n,m)}\log ({\rm v}_m(\xi\cap K))\,d\nu_{m}(\xi),&&q=0.
% \end{align*}
% The pioneering paper \cite{CLWX26} established the property that  the differentials of the aforementioned centro-sectional functional lead to the so-called centro-section measure on $S^{n-1}$ that satisfies the characteristic variational properties \eqref{spheri-var} and \eqref{spheri-var2}.

Finally, we consider two observations from calculus. First, if $|\tau|\leq \frac1{2m}$ for an integer $m\geq 1$, then
\begin{equation}
\label{1+tau-m}
1-m|\tau|\leq (1+\tau)^m\leq 1+2m|\tau|,
 \end{equation}
where the lower bound follows from the AM-GM inequality for m-1 copies of $1$ and one copy of $1-m|\tau|$, and the upper bound is a consequence of the estimate $1+\tau\leq e^\tau$. Secondly, derivation yields the following.

\begin{claim}
\label{tlogt-monotonicity}
The function $\tau\mapsto \tau\log\frac1{\tau}$ is monotone increasing for $\tau\in(0,\frac1{e}]$, and
$\lim_{\tau\to 0^+}\tau\log\frac1{\tau}=0$.
\end{claim}

\section{Convex functions}
\label{secConvexFunctions}

The main reference for this section is \cite{AGM15}. We say that a function $\psi:\Rn\to(-\infty,\infty]$ is proper if $\psi(x)<\infty$ for some $x\in\R^n$.
For any proper function $\psi:\Rn\to(-\infty,\infty]$ that is bounded from below, its Legendre transform $\psi^*:\Rn\to(-\infty,\infty]$ is the proper lower continuous convex function defined by the formula
\begin{equation*}
 \label{Legendre-def}
\psi^*(x)=\sup_{y\in\Rn}x\cdot y-\psi(y). 
\end{equation*}
  We observe that if $\alpha>0$ and $\beta\in\R$, then
\begin{equation}
\label{Legendre-shift}
(\alpha \psi+\beta)^*(x)=\alpha\psi^*\left(\frac{x}{\alpha}\right)-\beta,
\end{equation}
and if $\psi_1\leq\psi_2$ are proper and bounded from below, then
\begin{equation}
\label{Legendre-monotone}
\psi_1^*\geq \psi_2^*.
\end{equation}
For $\psi^{**}=(\psi^*)^*$, we have
\begin{equation}
\label{Legendre-stars} 
\begin{array}{lll}
\psi^{**}\leq &\psi&\mbox{ for any  proper function $\psi$ bounded from below,}\\[1ex]
\psi^{**}=&\psi&\mbox{ if $\psi$ itself is a lower continuous convex function.}
\end{array}
\end{equation}

Next, let $\psi:\Rn\to(-\infty,\infty]$ be a proper convex function. Its domain is ${\rm dom}\,\psi=\{x\in\R^n:\psi(x)<\infty\}$ that is readily convex, and its epigraph is the convex set ${\rm epi}\,\psi=\{(x,t)\in\R^{n+1}:x\in {\rm dom}\,\psi\;\&\;t\geq \psi(x)\}$. We say that the convex function $\psi$ is coercive if  $\lim_{\|x\|\to\infty}\psi(x)=\infty$; or equivalently, $\inf \psi>-\infty$ and the level sets $\{\psi(x)\leq s\}$ are bounded for any $s>\inf \psi$. 

\begin{claim}
\label{ConvexFunctionsBasic}
For any proper convex function $\psi:\Rn\to(-\infty,\infty]$, 
\begin{itemize}
\item if $x\in{\rm int}\,{\rm dom}\,\psi$ and $\psi(x)=s$, then $x\in\partial\{\psi\leq s\}$;
\item ${\rm epi}\,\psi$ is closed if and only if $\psi$ is lower semicontinuous;
\item $o\in{\rm int}\,{\rm dom}\,\psi$ if and only if $\psi^*$ is coercive;
\item if $\psi$ is coercive, then there exist $\alpha>0$ and $\beta\in \R$ such that $\psi(x)\geq \alpha\|x\|+\beta$ for $x\in\R^n$.
\end{itemize}
\end{claim}

For a convex function $\psi:\Rn\to(-\infty,\infty]$, if $\Omega={\rm int}\,{\rm dom}\,\psi\neq \emptyset$, then the set $\Omega'$ of $x\in\Omega$ where $\nabla\psi (x)$ exists is a Borel set and $\HH^n(\Omega\backslash\Omega')=0$; moreover,
\begin{equation}
\label{nablaphi-continuous}
\nabla\psi\mbox{ is continuous on }\Omega'.
\end{equation}

 If $K\subset\R^n$ is a non-empty convex set, then its support function $h_K:\R^n\to(-\infty,\infty]$ is defined by
\begin{equation}
\label{hKdef}
h_K(z)=\sup_{y\in K} z\cdot y,
\end{equation}
which is a proper convex one-homogeneous function; namely, $h_K(\lambda z)=\lambda h_K(z)$ for $z\in\R^n$ and $\lambda>0$. We observe that we can write maximum instead of supremum in \eqref{hKdef} if $K$ is compact, and in this case, for any $x\in\partial K$, there exists an exterior unit vector $u\in S^{n-1}$ such that $h_K(u)=u\cdot x$. How support function is related to Legendre transform is discussed in Example~\ref{example-c1L}.   For more details on convex geometry, refer to  Schneider's book ~\cite{S14}.

Let $K\subset\R^n$ be a  convex set with ${\rm int}\,K\neq\emptyset$. If $K\neq\R^n$ (and hence $\HH^{n-1}(\partial K)>0$), then the set $\partial'K$ of $x\in\partial K$ where there exists a unique exterior unit normal $u_{K}(x)$ to $K$ at $x$
is a Borel set $\partial'K\subset\partial K$ such that $\HH^{n-1}(\partial K\backslash \partial'K)=0$; moreover,
\begin{equation}
\label{uK-continuous}
u_K\mbox{ is continuous on }\partial' K.
\end{equation}
The associated recession cone is $\Sigma_K=\{x\in \R^n:y+\lambda x\in K \;\forall \lambda\geq 0 \mbox{ and }y\in{\rm int}\,K\}$, and its dual cone is $\Sigma_K^*=\{x\in \R^n:x\cdot y\leq 0 \;\forall y\in \Sigma_K\}$ where both $\Sigma_K$ and $\Sigma_K^*$ are closed convex cones, and $\Sigma_K\neq \{o\}$ if and only if $K$ is unbounded. We observe that ${\rm int}\,\Sigma_K^*\subset {\rm dom}\,h_K\subset \Sigma_K^*$, and  in particular, $u_K(x)\in {\rm dom}\,h_K$ for any $x\in \partial'K$.

If $\psi$ is a convex function on $\R^n$ with $\psi(o)<\infty$, then we define its recession function (or horizon function) for $x\in\R^n$ by the formula
$$
\bar{\psi}(x)=\lim_{r\to \infty}\frac{\psi(rx)}r.
$$
Then $\bar{\psi}$ is a well-defined convex  one-homogeneous function, and its epigraph is the recession cone of the epigraph of $\psi$. We note that by abusing notation, we write $\bar{\psi}^*$ to denote $\overline{\psi^*}$. As examples in the case of a proper convex function $\psi$, we note that
\begin{align}
\label{barpsi-hK}
\bar{\psi}=&h_K  \mbox{ \ if $\psi=h_K$ for a convex set $K\subset\R^n$},\\
\label{barpsi-dompsi-bounded}
\bar{\psi}^*=& h_F \mbox{ \ if $F={\rm dom}\,\psi$ is bounded.}
\end{align}

A crucial notion of our paper is epiconvergence of lower continuous coercive convex functions. All related notions, notations and statements in the rest of this section are taken from Li, Mussnig \cite{LM22}.
We denote the family of proper lower continuous coercive convex functions on $\R^n$ by ${\rm Conv}_c(\R^n)$, and the family of elements $\varphi \in {\rm Conv}_c(\R^n)$ with ${\rm int}\,{\rm dom}\,\varphi\neq \emptyset$ by ${\rm Conv}_c^n(\R^n)$.
We note that level sets of a convex function are coming from the ``horizontal sections" of its epigraph, hence the name  \emph{epiconvergence}.

\begin{definition}
\label{epiconvergence-def}
For $\varphi_k,\varphi\in {\rm Conv}_c(\R^n)$, $k\in\N$,
 $\varphi_k$ epiconverges to $\varphi$, in notation, 
 $$
 \varphi_k\xrightarrow{\rm epi} \varphi,
 $$
 if and only if for any $s>\min \varphi$, the level sets $\{\varphi_k\leq s\}$ tend to $\{\varphi\leq s\}$, and for all $s<\min \varphi$, the level set $\{\varphi_k\leq s\}$ is the empty set for large $k$.
\end{definition}

If ${\rm int}\,{\rm dom}\,\varphi\neq\emptyset$, then additional natural characterizations of epiconvergence are available.

\begin{lemma}
\label{epiconvergence-interior}
If $\varphi_k,\varphi\in {\rm Conv}_c^n(\R^n)$, then the following three statements are equivalent. 
\begin{itemize}
\item $\varphi_k\xrightarrow{\rm epi} \varphi$.
\item $\lim_{k\to\infty} \varphi_k(x)=\varphi(x)$ holds for any $x\not\in \partial ({\rm dom}\,\varphi)$.
\item For any compact $X\subset \R^n$ with $X\cap\partial ({\rm dom}\,\varphi)=\emptyset$, $\varphi_k|_X$ tends uniformly to $\varphi|_X$.
\end{itemize}
\end{lemma}

Epiconvergence implies that all epigraphs are contained in the same cone with axial rotational symmetry.

\begin{lemma}
\label{epiconvergence-property}
If $\varphi,\varphi_k\in {\rm Conv}_c^n(\R^n)$, $k\in\N$, and $\varphi_k\xrightarrow{\rm epi} \varphi$, then there exist $\alpha>0$ and $\beta\in\R$ such that $\varphi(x),\varphi_k(x)\geq \alpha\|x\|+\beta$ for any $x\in\R^n$ and $k\in\N$.
\end{lemma}

In turn, the analogue of the Blaschke Selection theorem is as follows. 

\begin{lemma}
\label{epiconvergence-selection}
If for some $\alpha>0$, $\beta,M\in\R$ the functions $\varphi_k\in {\rm Conv}_c(\R^n)$, $k\in\N$, satisfy that for any $k\in\N$, we have
\begin{itemize}
\item $\min\varphi_k\leq M$, and
\item $\varphi_k(x)\geq \alpha\|x\|+\beta$ for $x\in\R^n$,
\end{itemize}
then  $\varphi_{k'}\xrightarrow{\rm epi} \varphi$ holds for a subsequence $\{\varphi_{k'}\}$ and a $\varphi\in {\rm Conv}_c(\R^n)$.
\end{lemma}

\section{The coarea formula and the centro-sectional variational measures for log-concave functions}
\label{secCoarea}

We recall  that a function $f:\Rn\to[0,\infty)$ is log-concave if $f((1-\lambda)x+\lambda y)\geq f(x)^{1-\lambda}f(y)^\lambda$ for any $x,y\in\Rn$ and $\lambda\in[0,1]$; or equivalently, $f=e^{-\varphi}$ for a convex function $\varphi:\Rn\to(-\infty,\infty]$. It is well-known that $0<\int_{\Rn} f<\infty$ if and only if $0<\sup f<\infty$, and the level sets $\{f\geq s\}$ are bounded with non-empty interior whenever $0<s<\sup f$. The domain of such $f$ is the convex set $D_f=\{f>0\}$ with non-empty interior. If $f$ is upper semicontinuous, then the level sets $\{f\geq s\}$ are convex bodies for $s\in(0,\sup f)$.

If $f$ is a log-concave function on $\R^n$ such that $0<\int_{\Rn} f<\infty$ and $o\in{\rm int}\,D_f$, then Claim~\ref{ConvexFunctionsBasic} yields the existence of $\alpha>0$, $\beta\in \R$ and $r_0,s_0>0$ such that
\begin{align}
\label{log-concave-f-exp-above}
f(x)\leq &e^{-\alpha\|x\|-\beta} &&\mbox{for }x\in\R^n,\\
\label{log-concave-f-ball-in-lvelset}
r_0B^n\subset&\{f\geq s\} &&\mbox{for }s\in(0,s_0].
\end{align}
In particular, for any $m\in\{1,\ldots,n-1\}$, there exist $A_1>A_0>1$ depending on $m$ and $f$ such that if $\xi\in{\rm G}(n,m)$, then
\begin{equation}
\label{log-concave-f-xi-A0A1}
A_0\leq ({\rm R}_mf)\leq A_1.
\end{equation}

For $t>0$, the sup-convolution $f\oplus t\cdot g$ of log-concave functions $f=e^{-\varphi}$ and $g=e^{-\psi}$ on $\Rn$ is the upper semicontinuous log-concave function
\begin{equation*}
	 (  f\oplus t\cdot g )  (z)= \sup_{x+t y=z} f(x)g(y)^t,
\end{equation*}
which can be written in the form
$$
f\oplus t\cdot g=e^{-(\varphi^*+t\psi^*)^*}.
$$

\begin{example} 
\label{example-c1L}
If $c>0$ and $L\subset \R^n$ is a compact convex set, then $c\mathbf{1}_L$ is an upper semicontinuous log-concave function, and $c \mathbf{1}_L=e^{-\psi}$ where
$$
\psi(x)=\left\{
\begin{array}{ll}
-\log c&\mbox{ if }x\in L,\\
\infty &\mbox{ if }x\in \R^n\backslash L.
\end{array}\right.
$$
For the Legendre transform and recession function, it follows that
\begin{align*}
\psi^*=&h_L+\log c,\\
\bar{\psi}^*=&h_L.
\end{align*}
Now, let $f=e^{-\varphi}$ be a log-concave function  on $\Rn$ such that $0<\int_{\Rn} f<\infty$, and let $F_s=\{f\geq s\}$ if $0<s\leq \sup f$, and hence $f\oplus t\cdot (c\mathbf{1}_L)=e^{-(\varphi^*+t\psi^*)^*}$ holds  for $t\geq 0$. In addition, if $t\in\R$, then
\begin{equation}
\label{perturb-with-c1L}
e^{-(\varphi^*+t\psi^*)^*}=c^t\tilde{f}_t \mbox{ \ for }\tilde{f}_t=e^{-(\varphi^*+th_L)^*},
\end{equation}
where, according to \cite{FR26}, we have (cf. \eqref{FoplustL})
\begin{align}
\nonumber
\sup \tilde{f}_t\leq &\sup f&&\mbox{with equality if }t>0,\\
\label{log-sum-level}
\{\tilde{f}_t\geq s\}=&F_s\oplus tL&&\mbox{if }0<s\leq \sup f.
\end{align}
\end{example}

Let $q\in \R$ and $m=\{1,\ldots,n-1\}$. As mentioned in the introduction, the centro-sectional functional $\widetilde{\Psi}_{m,q}(K)$ was introduced in \cite{CLWX26} for convex bodies $K\subset\mathbb{R}^n$ containing the origin in their interiors. Inspired by this work, we introduce an analogue for functions.
% \begin{align*}
% \widetilde{\Psi}_{m,q}(K)=&\int_{{\rm G}(n,m)}{\rm v}_m(\xi\cap K)^q\,d\nu_{m}(\xi),&& q\neq 0,\\
% \widetilde{\Psi}_{m,0}(K)=&\int_{{\rm G}(n,m)}\log ({\rm v}_m(\xi\cap K))\,d\nu_{m}(\xi),&&q=0.
% \end{align*}

\begin{definition}
\label{Psimqfdef-sectionint}
If  $q\in\R$, $m=1,\ldots,n-1$, and $f$ is a log-concave function  on $\Rn$ such that $0<\int_{\Rn} f<\infty$ and $o\in{\rm int}\,D_f$, then we define
\begin{equation}
\label{Psimqfdef-sectionint-eq}
\begin{array}{rcll}
\widetilde{\Psi}_{m,q}(f)&=\int_{{\rm G}(n,m)}\|f|_{\xi}\|_1^qd\nu_{m}(\xi)&\mbox{ if }  q \neq 0, \\
\widetilde{\Psi}_{m,0}(f)&=\int_{{\rm G}(n,m)}\log \|f|_{\xi}\|_1 d\nu_{m}(\xi) &\mbox{ if }  q = 0,
\end{array}
\end{equation}
where the integrals in \eqref{Psimqfdef-sectionint-eq} are 
finite according to \eqref{log-concave-f-xi-A0A1}.
\end{definition}

For any $m\in\{1,\ldots,n-1\}$ and  log-concave function $f$ on $\Rn$ such that $0<\int_{\Rn} f<\infty$ and $o\in{\rm int}\,D_f$, Keith Ball defined a convex body $K_m(f)$ with $o\in{\rm int}\,K_m(f)$ that satisfies 
\begin{equation}
\label{Ball-body}
\varrho_{K_m(f)}(u)^m=\frac{m}{f(o)}\int_0^\infty r^{m-1} f(ru)\,dr
\end{equation}
for $u\in S^{n-1}$ where the  boundedness of $K_m(f)$ follows from 
\eqref{log-concave-f-exp-above}.
We recall that the measure of the empty set is set to be zero.

\begin{lemma}
\label{fxi-representations}
Let $m\in\{1,\ldots,n-1\}$.  For any  log-concave function $f$ on $\Rn$ such that $0<\int_{\Rn} f<\infty$ and $o\in{\rm int}\,D_f$, if $\xi\in {\rm G}(n,m)$, then
\begin{align}
\label{fxi-representations-Ball-eq}
{\rm R}_mf(\xi)=& \left( \widetilde{\mathcal{R}}_m\int_0^\infty r^{m-1} f(ru)\,dr\right)(\xi)= f(o)\HH^m\left(K_m(f)\cap \xi\right),\\
 \label{fxi-representations-area-levelset-eq}
 =& \int_0^\infty\HH^m\left(\{f> s\}\cap\xi\right)\,ds,\\
 \label{fxi-representations-Rmlevelset-eq}
 =& \frac1{m}\int_0^\infty\int_{S^{n-1}\cap\xi}\left({\rm sign}\,\varrho_{\{f> s\}}\right)\left|\varrho_{\{f> s\}}\right|^m\,d\HH^{m-1}\,ds,
\end{align}
where within the formula  \eqref{fxi-representations-Ball-eq}, the expression $\int_0^\infty r^{m-1} f(ru)\,dr$  is a short hand notation for the bounded function $u\mapsto \int_0^\infty r^{m-1} f(ru)\,dr$ of $u\in S^{n-1}$.

In addition, the function $\xi\mapsto {\rm R}_mf(\xi)$ of $\xi\in {\rm G}(n,m)$ is continuous.
 \end{lemma}
\noindent{\bf Remark.} The boundedness of the function $u\mapsto \int_0^\infty r^{m-1} f(ru)\,dr$ of $u\in S^{n-1}$ follows from \eqref{log-concave-f-exp-above}.
\begin{proof} For \eqref{fxi-representations-Ball-eq}, integration in polar coordinates yields that
\begin{align*}
{\rm R}_mf=&  \int_{S^{n-1}\cap\xi}\int_0^\infty r^{m-1} f(ru)\,dr\,d\HH^{m-1}(u)\\
=&\frac{f(o)}{m}\int_{S^{n-1}\cap\xi}\varrho_{K_m(f)}(u)^m\,d\HH^{m-1}(u)=f(o)\HH^m\left(K_m(f)\cap \xi\right).
\end{align*}
It follows from \eqref{fxi-representations-Ball-eq} that the function $\xi\mapsto ({\rm R}_mf)$ is continuous.

Next, \eqref{fxi-representations-area-levelset-eq} is a direct consequence of the layer cake formula, and in turn \eqref{polar-volume} and \eqref{fxi-representations-area-levelset-eq} imply \eqref{fxi-representations-Rmlevelset-eq}.
\end{proof}
 By Lemma \ref{fxi-representations}, we have the following result.
\begin{coro}
\label{Psimqf-Ball-body}
If  $m\in\{1,\ldots,n-1\}$ and $f$ is a  log-concave function on $\Rn$ such that $0<\int_{\Rn} f<\infty$ and $o\in{\rm int}\,D_f$, then
\begin{align}
\label{Psimqf-Ball-body-eq}
 \widetilde{\Psi}_{m,q}(f)=&f(o)^q\widetilde{\Psi}_{m,q}(K_m(f))&&\mbox{if }q\neq 0,\\
\label{Psimqf-Ball-body0-eq}
 \widetilde{\Psi}_{m,0}(f)=&\log f(o)+\widetilde{\Psi}_{m,0}(K_m(f))&&\mbox{if }q=0.
\end{align}
\end{coro}

According to \cite{BLY26+}, if $q\in \R$,  $m\in\{1,\ldots,n-1\}$ and $F_k,F\subset\R^n$, $k\in\N$, are convex bodies containing the origin in their interior, and $F_k$ tends to $F$, then
\begin{equation}
\label{PsimqF-continuous}
\lim_{k\to\infty} \widetilde{\Psi}_{m,q}(F_k)=\widetilde{\Psi}_{m,q}(F).
\end{equation}
Now, we show that $\widetilde{\Psi}_{m,q}(f)$ is continuous as a function of a log-concave $f$.

\begin{lemma}
\label{Psimqf-continuous2}
Let $q \in \R$,  $m\in\{1,\ldots,n-1\}$.  If $\varphi_k\xrightarrow{\rm epi} \varphi$ holds for $\varphi,\varphi_k\in {\rm Conv}_c(\R^n)$, $k\in\N$, and $o\in{\rm int}\,{\rm dom}\,\varphi$, then the log-concave functions $f=e^{-\varphi}$ and $f_k=e^{-\varphi_k}$ satisfy that
\begin{equation}
\label{Psimqf-continuous-eq}
\lim_{k\to\infty} \widetilde{\Psi}_{m,q}(f_k)=\widetilde{\Psi}_{m,q}(f).
\end{equation}
\end{lemma}
\begin{proof}
We deduce from Lemma~\ref{epiconvergence-property} that there exist $\alpha>0$ and $\beta\in\R$ such that 
\begin{equation}
\label{Psimqf-continuous-fkf-above}
f(x),f_k(x)\leq e^{-(\alpha\|x\|+\beta)} \mbox{ \ for any $x\in\R^n$ and $k\in\N$.}
\end{equation}
In addition, Lemma~\ref{epiconvergence-interior} yields that 
\begin{equation}
\label{Psimqf-continuous-fkf-limit}
\lim_{k\to\infty} f_k(x)=f(x) \mbox{ \ holds for any $x\not\in \partial D_f$.}
\end{equation}
First, we show that for any $u\in S^{n-1}$, we have $\lim_{k\to\infty} \varrho_{K_m(f_k)}(u)=\varrho_{K_m(f)}(u)$, which statement is equivalent with (cf. \eqref{Ball-body} and \eqref{Psimqf-continuous-fkf-limit})
\begin{equation}
\label{Psimqf-continuous-radial-limit}
\lim_{k\to\infty} \int_0^\infty r^{m-1} f_k(ru)\,dr=\int_0^\infty r^{m-1} f(ru)\,dr.
\end{equation}
For Lebesgue's Dominant Convergence theorem, the integrable pointwise upper bound follows from \eqref{Psimqf-continuous-fkf-above}, and the pointwise limit is a consequence of  \eqref{Psimqf-continuous-fkf-limit}, concluding the proof of \eqref{Psimqf-continuous-radial-limit}.

Since $\lim_{k\to\infty} \varrho_{K_m(f_k)}(u)=\varrho_{K_m(f)}(u)$ for any $u\in S^{n-1}$, we deduce that $K_m(f_k)$ tends to $K_m(f)$ by Claim~\ref{convergence-by-radial}, and hence Lemma~\ref{Psimqf-continuous2} follows from Corollary~\ref{Psimqf-Ball-body} and \eqref{Psimqf-continuous-fkf-limit}.
\end{proof}

In order to define our centro-sectional variational measures $\widetilde{A}_{m,q}^{e}(f,\cdot)$ on $\R^n$ and $\widetilde{A}_{m,q}^{s}(f,\cdot)$ on $S^{n-1}$  for an upper semicontinuous log-concave function $f=e^{-\varphi}$, $0<\int_{\Rn} f<\infty$ and $M=\sup f$, we recall the coarea formula.
For $\Omega={\rm int}\,D_f$,
 let  $\Omega'\subset\Omega$ be the Borel set of $x\in\Omega$, where $\nabla\varphi(x)$ exists, and hence $\HH^n(\Omega\backslash\Omega')=0$, and  
$\nabla f(x)=-f(x)\cdot\nabla\varphi(x)$ 
for $x\in\Omega'$, where $\nabla f(x)$ is continuous on $\Omega'$. If $\Upsilon$ is a Borel function on $\Omega$ that is either non-negative or bounded with compact support, then Federer's coarea formula says that
\begin{equation}
\label{coarea-Federer}
\int_0^M\int_{\Omega\cap \partial \{f>s\}}\Upsilon(x)\,d\HH^{n-1}(x)\,ds=\int_{\Omega} \Upsilon\cdot f\cdot \|\nabla\varphi\|\,d\HH^{n}.
\end{equation}

Let us consider two consequences of \eqref{coarea-Federer}. We recall that the coarea formula for  a BV function $f$ ($f$ is of bounded variation) says (see, for example, Section~5.1 in \cite{BFR26})
\begin{equation}
\label{coarea-formula-general}
\int_{\R^n}\Upsilon\, d|D f|=\int_{\R}
\int_{\partial^* \{f>t\}}\,\Upsilon\,d\mathcal H^{n-1}\,dt,
\end{equation} 
where $\Upsilon$ is a Borel measurable function on $\R^n$. However, in our special case when $f$ is  an upper semicontinuous log-concave function  on $\Rn$ with $0<\int_{\Rn} f<\infty$, we can prove the coarea formula \eqref{coarea-formula-general} in a slighly more precise form using \eqref{coarea-Federer} and the fact that the level sets are compact  convex sets.

\begin{lemma}[Coarea formula for log-concave functions]
\label{coarea-BV}
If $f=e^{-\varphi}$ is an upper semicontinuous log-concave function  on $\Rn$ with $0<\int_{\Rn} f<\infty$ and $M=\sup f$, and $\Upsilon$ is a Borel measurable function on $\R^n$ that is either non-negative or bounded, then 
\begin{equation}
\label{coarea-BV-eq}
\int_0^M\int_{\partial \{f\geq s\}}\Upsilon(x)\,d\HH^{n-1}(x)\,ds=\int_{D_f}\Upsilon f \|\nabla\varphi\|\,d\HH^{n}+
\int_{\partial D_f}\Upsilon f\,d\HH^{n-1},
\end{equation}
where all the three integrals are finite if $\Upsilon$ is bounded, and assuming that $\Upsilon\geq 0$, the left hand side of \eqref{coarea-BV-eq} is finite if and only if both integrals on the right hand side of \eqref{coarea-BV-eq} are finite. 
\end{lemma}
\noindent{\bf Remark.} In particular, we have
\begin{equation}
\label{Dario-Boaz-Liran-eq}    
\int_{\R^n} f\cdot \|\nabla \varphi\|\,dx<\infty
\mbox{ \ and \ }
\int_{\partial D_f} f\,d\HH^{n-1}<\infty,
\end{equation}
where the first inequality in \eqref{Dario-Boaz-Liran-eq} is due to 
Cordero-Erausquin, Klartag \cite{CK15}, while the second is due to Rotem \cite{R23}, who used more involved arguments.
\begin{proof}
Let $F_s=\{f\geq s\}$ for $s\in(0,M)$ that is a convex body.
For any $x\in\partial D_f$, we have $f(x)=\max\{s\in(0,M]: x\in \partial F_s\}$.  Thus for any $s\in(0,M)$, applying \eqref{coarea-Federer} for $\partial \{f\geq s\}\cap({\rm int}\,D_f)$, and Fubini's theorem in $\R\oplus \partial D_f$ for $\partial \{f\geq s\}\cap(\partial D_f)$ yields \eqref{coarea-BV-eq}. It also follows that assuming that $\Upsilon\geq 0$, the left hand side of \eqref{coarea-BV-eq} is finite if and only if both integrals on the right are finite.

Therefore, it remains only to prove that $\int_0^M\HH^{n-1}(\partial F_s)\,ds<\infty$. By the monotonicity of the surface area and $F_s\subset F_{s_0}$ for $s<s_0$, this is equivalent to verify that 
\begin{equation}
\label{coarea-BV-finiteness}
\int_0^{\min\{M,e^{-1}\}}\HH^{n-1}(\partial F_s)\,ds<\infty.
\end{equation}
We recall from \eqref{log-concave-f-exp-above} that there exist $\alpha>0$ and $\beta\in \R$ such that
$f(x)\leq e^{-\alpha\|x\|-\beta}$ for $x\in\R^n$, and hence if $s\in(0,e^{-1})$ and $x\in \partial F_s$, then
$\|x\|\leq \frac{|\beta|}{\alpha}+\frac{1}{\alpha}\log\frac1{s}\leq \gamma \log\frac1{s}$ for $\gamma=\frac{|\beta|+1}{\alpha}$. Thus $\HH^{n-1}(\partial F_s)\leq (n-1)\omega_{n-1}\gamma^{n-1}(\log\frac1{s})^{n-1}$ for $s\in(0,e^{-1}]$, and as $(\log\frac1{s})^{n-1}$ is integrable on $(0,e^{-1}]$, we conclude \eqref{coarea-BV-finiteness}.
\end{proof}

 Another consequence of \eqref{coarea-Federer} is (i) in Claim~\ref{exterior-normal}.

\begin{claim}
\label{exterior-normal}
Let $f=e^{-\varphi}$ be an upper semicontinuous log-concave function  on $\Rn$ with $0<\int_{\Rn} f<\infty$ and $M=\sup f$. For $\HH^1$ a.e. $s\in (0,M)$, 
\begin{enumerate}
\item[(i)]  for $\HH^{n-1}$ a.e. $x\in (\partial \{f\geq s\})\cap({\rm int}\,D_f)$, the gradient $\nabla\varphi(x)$ exists and $\frac{\nabla\varphi(x)}{\|\nabla\varphi(x)\|}$ is the unique exterior unit normal at $x$ to $\{f\geq s\}$;
\item[(ii)] for $\HH^{n-1}$ a.e. $x\in (\partial \{f\geq s\})\cap (\partial D_f)$, there exists a unique exterior unit normal $u$ at $x$ to $\{f\geq s\}$, and it is the exterior unit normal at $x$ to ${\rm cl}\,D_f$.
\end{enumerate}
\end{claim}

Motivated by the coarea formula \eqref{coarea-BV-eq}, and following the footsteps of \cite{FR26} and \cite{HLXZ24}, we can now rephrase the definition of the centro-sectional variational measures in Definition \ref{tildeAdef} in terms of integrals. 

For the spherical measure $\widetilde{A}_{m,q}^{s}(f,\cdot)$ on $S^{n-1}$, we also provide a representation in terms of integrals on the sphere $S^{n-1}$. For this, let  $K\subset\R^n$ be a  convex set with $o\in{\rm int}\,K$. In this case,  the associated recession cone is $\Sigma_K=\{x\in K:\lambda x\in K \;\forall \lambda\geq 0\}$ (cf. Section~\ref{secPreliminaries}), and we also consider
$$
\Theta_K=\left\{\frac{x}{\|x\|}:x\in \partial K\right\}=S^{n-1}\backslash \Sigma_K,
$$
that is a non-empty open subset of $S^{n-1}$ where $\Theta_K=\emptyset$ if $K=\R^n$ and $\Theta_K=S^{n-1}$ if $K$ is compact. For $u\in\Theta_K$, we have an $x\in \partial K$ such that $u=\frac{x}{\|x\|}$, and we define
\begin{align*}
\varrho_K(u)=&\|x\|\mbox{ \ and \ }
r_K(u)=x=\varrho_K(u)u,\\
\alpha_K(u)=&u_K(x)\mbox{ \ provided that }x\in\partial'K,
\end{align*}
and hence $h_K(\alpha_K(u))=u_K(x)\cdot x$ provided that $x\in\partial'K$.
We observe that $r_K(u)\in\partial'K$ for $\HH^{n-1}$ a.e. $u\in\Theta_K$ as the radial projection from $\partial D_f$ to $S^{n-1}$ is locally Lipschitz.
If $K\neq \R^n$ and $\Upsilon: \partial K\to[0,\infty)$ is Borel measurable, then \cite{HLYZ16} establishes
\begin{equation}
\label{HLYZ-cone-volume0}
\int_{\partial K}\Upsilon(x)h_K(u_K(x))\,d\HH^{n-1}(x)=
\int_{\Theta_K}\Upsilon\left(r_K(u)\right)\varrho_K(u)^n\,d\HH^{n-1}(u).
\end{equation}
Here \cite{HLYZ16} only considers the case when $K$ is compact, but this case easily generalizes to the non-compact case by first considering functions $\Upsilon$ with compact support in \eqref{HLYZ-cone-volume0}. Combining \eqref{tildeAdef-s} and \eqref{HLYZ-cone-volume0}, we have the representation \eqref{tildeAdef-s-sphere} for  the spherical centro-sectional measure $\widetilde{A}_{m,q}^{s}(f,\cdot)$.

\begin{lemma}
\label{tildeAdef-integrals}
Let  $q\in \R$ and $m\in\{1,\ldots,n-1\}$. Let $f$ be an upper semicontinuous log-concave function  on $\Rn$ such that $0<\int_{\Rn} f<\infty$ and  $o\in{\rm int}\,D_f$. If $\Upsilon:\R^n\to\R$ and $\zeta:S^{n-1}\to\R$ are Borel measurable and either bounded or non-negative, then
\begin{align}
\label{tildeAdef-Rn}
\int_{\R^n}\Upsilon\,d\widetilde{A}_{m,q}^{e}(f,\cdot)=&\int_{\R^n}\Upsilon(\nabla\varphi(x)) \|x\|^{m-n}f(x) \left(\widetilde{\mathcal{R}}^*_m({\rm R}_mf)^{q-1}\right)\left(\frac{x}{\|x\|}\right)\,dx,\\
\label{tildeAdef-s}
\int_{S^{n-1}}\zeta\,d\widetilde{A}_{m,q}^{s}(f,\cdot)=&\int_{\partial D_f}\zeta(u_{D_f}(x))\|x\|^{m-n}f(x) \left(\widetilde{\mathcal{R}}^*_m({\rm R}_mf)^{q-1}\right)\left(\frac{x}{\|x\|}\right)\,dx\\
\label{tildeAdef-s-sphere}
=&\int_{S^{n-1}}(\zeta\circ\alpha_{D_f})\cdot \frac{\varrho_{D_f}^{m}\cdot\mathbf{1}_{\Theta_{D_f}}}{h_{D_f}\circ\alpha_{D_f}}\cdot (f\circ r_{D_f}) \cdot\widetilde{\mathcal{R}}^*_m({\rm R}_mf)^{q-1}\,d\HH^{n-1}.
\end{align}
\end{lemma}
\noindent{\bf Remark.} It follows that 
if $q\neq 0$, $m\in\{1,\ldots,n-1\}$, $f$ is an  upper semicontinuous log-concave function  on $\Rn$ such that $0<\int_{\Rn} f<\infty$ and  $o\in{\rm int}\,D_f$ and $\lambda>0$, then
\begin{equation}
\label{Psi-Amq-homogeneity}
\begin{array}{rcl}
\widetilde{\Psi}_{m,q}(\lambda f)&=&\lambda^q\,\widetilde{\Psi}_{m,q}(f),\\[1ex]
\widetilde{A}_{m,q}^{e}(\lambda f,\cdot)&=&\lambda^q\,\widetilde{A}_{m,q}^{e}(f,\cdot) \mbox{ \ and \ }
\widetilde{A}_{m,q}^{s}(\lambda f,\cdot)=\lambda^q\,\widetilde{A}_{m,q}^{s}(f,\cdot).
\end{array}
\end{equation}

Let us discuss the finiteness  properties of the centro-sectional variational measures $\widetilde{A}_{m,q}^{e}(f,\cdot)$ and $\widetilde{A}_{m,q}^{s}(f,\cdot)$.

\begin{prop}
\label{tildeA-finiteness}
Let  $q\in \R$ and $m\in\{1,\ldots,n-1\}$. Let $f$ be an upper semicontinuous log-concave function  on $\Rn$ such that $0<\int_{\Rn} f<\infty$ and  $o\in{\rm int}\,D_f$. Then the associated centro-sectional variational Borel measures satisfy that ${\rm supp}\,\widetilde{A}_{m,q}^{e}(f,\cdot)\subset{\rm cl}\,D_f$,
\begin{align}
\label{tildeA-finiteness-Rn-moment}
\int_{\Rn}\|x\|\,d\widetilde{A}_{m,q}^{e}(f,x)<&\infty&&\mbox{for }q\in\R, \\
\label{tildeA-finiteness-RntildePsi}
\widetilde{A}_{m,q}^{e}(f,\R^n)=\widetilde{\Psi}_{m,q}(f)<&\infty &&\mbox{if }q\neq 0, \\
\label{tildeA-finiteness-RntildePsi0}
\widetilde{A}_{m,0}^{e}(f,\R^n)=&1, \\
\label{tildeA-finiteness-RntildePsis}
\widetilde{A}_{m,q}^{s}(f,S^{n-1})<&\infty&&\mbox{for }q\in\R. 
\end{align}
\end{prop}
\begin{proof} 
For $q\in \R$, the estimate $\widetilde{\Psi}_{m,q}(f)<\infty$ follows from \eqref{log-concave-f-xi-A0A1}.
According to Definition~\ref{tildeAdef}, we have
\begin{align*}
\widetilde{A}_{m,q}^{e}(f,\Rn)
=&\int_{D_f} \|x\|^{m-n}f(x) \left(\widetilde{\mathcal{R}}^*_m({\rm R}_mf)^{q-1}\right)\left(\frac{x}{\|x\|}\right)\,dx,\\
\widetilde{A}_{m,q}^{s}(f,S^{n-1})=
&\int_{\partial D_f}\|x\|^{m-n}f(x) \left(\widetilde{\mathcal{R}}^*_m({\rm R}_mf)^{q-1}\right)\left(\frac{x}{\|x\|}\right)\,d\HH^{n-1}(x).
\end{align*}

We write $d\xi$ for $\xi\in{\rm G}(n,m)$ to abbreviate $d\nu_{m}(\xi)$. To prove  \eqref{tildeA-finiteness-RntildePsi}, for $q\neq 0$, we deduce
first from \eqref{fxi-representations-Ball-eq} in Lemma~\ref{fxi-representations}, then from the duality formula \eqref{Rm-dualRm-sphere}, and finally from using polar coordinates that
\begin{align*}
\widetilde{\Psi}_{m,q}(f)=& \int_{{\rm G}(n,m)}({\rm R}_mf)^{q-1}\widetilde{\mathcal{R}}_m\int_0^\infty r^{m-1} f(ru)\,dr\, d\xi\\
=& \int_{S^{n-1}}\left(\widetilde{\mathcal{R}}_m^*({\rm R}_mf)^{q-1}\right)(u)\int_0^\infty r^{m-1} f(ru)\,dr\,d\HH^{n-1}(u)\\
=&\int_{\R^n} \|x\|^{m-n}f(x) \left(\widetilde{\mathcal{R}}^*_m({\rm R}_mf)^{q-1}\right)\left(\frac{x}{\|x\|}\right)\,dx
=\widetilde{A}_{m,q}^{e}(f,\R^n).
\end{align*} 

Now, we prove \eqref{tildeA-finiteness-RntildePsi0}. Using again polar coordinates, the duality formula \eqref{Rm-dualRm-sphere}, and \eqref{fxi-representations-Ball-eq} in Lemma~\ref{fxi-representations},  we deduce that
\begin{align*}
\widetilde{A}_{m,0}^{e}(f,\R^n)=&\int_{\R^n}\|x\|^{m-n}f(x) \left(\widetilde{\mathcal{R}}^*_m({\rm R}_mf)^{-1}\right)\left(\frac{x}{\|x\|}\right)\,dx \\
=& \int_{S^{n-1}}\left(\widetilde{\mathcal{R}}_m^*({\rm R}_mf)^{-1}\right)(u)\int_0^\infty r^{m-1} f(ru)\,dr\,d\HH^{n-1}(u)\\
=&\int_{{\rm G}(n,m)}({\rm R}_mf)^{-1}\widetilde{\mathcal{R}}_m\int_0^\infty r^{m-1} f(ru)\,dr\,d\nu_{m}\\
=&\int_{{\rm G}(n,m)}({\rm R}_mf)^{-1}({\rm R}_mf)\,d\nu_{m}
=1.
\end{align*} 

 Next, for $q\in \R$, let $r_*>0$ such that $r_*B^n\subset {\rm int}\,D_f$, and it follows from \eqref{log-concave-f-xi-A0A1} that if $u\in S^{n-1}$, then
\begin{equation*}
\label{dualtildeRm-bound-Afinite}
\widetilde{\mathcal{R}}^*_m({\rm R}_mf)^{q-1}(u)\leq \frac{m\omega_m}{n\omega_n}\cdot  \max\left\{A_0^{q-1}, A_1^{q-1}\right\}:=A.
\end{equation*} 
In the case of the spherical measure (cf. \eqref{tildeA-finiteness-RntildePsis}), we have
\begin{align*}
\widetilde{A}_{m,q}^{s}(f,S^{n-1})&\leq A\int_{\partial D_f}\|x\|^{m-n} f(x)\,d\HH^{n-1}(x)\\
\quad\leq & Ar_*^{m-n}\int_{\partial D_f}f(x)\,d\HH^{n-1}(x)<\infty
\end{align*}
by $1\leq m\leq n-1$ and \eqref{Dario-Boaz-Liran-eq} after Lemma~\ref{coarea-BV}.

Turning to the first moment of $\widetilde{A}_{m,q}^{e}(f,\cdot)$ (cf. \eqref{tildeA-finiteness-Rn-moment}), let $f=e^{-\varphi}$ for a lower semicontinuous convex function $\varphi$.
One finds some $M_*>0$ such that $\|\nabla f(x)\|=f(x)\|\nabla\varphi(x)\|\leq M_*$ for any $x\in r_*B^n$ where $\nabla\varphi(x)$ exists, thus \eqref{tildeAdef-Rn} yields that
\begin{align*}
\int_{\R^n}\|z\|\,d\widetilde{A}_{m,q}^{e}(f,z)=&\int_{D_f}\|\nabla\varphi(x)\|\cdot \|x\|^{m-n}f(x) \left(\widetilde{\mathcal{R}}^*_m({\rm R}_mf)^{q-1}\right)\left(\frac{x}{\|x\|}\right)\,dx\\
\leq & 
A\int_{D_f} \|x\|^{m-n} \|\nabla f(x)\|\,dx\\
\leq &AM_*\int_{r_*B^n} \|x\|^{m-n} \,dx+
Ar_*^{m-n}\int_{\R^n}\|\nabla f(x)\|\,dx<\infty
\end{align*} 
by $1\leq m\leq n-1$ and \eqref{Dario-Boaz-Liran-eq} after Lemma~\ref{coarea-BV}, completing the proof of  Proposition~\ref{tildeA-finiteness}.
\end{proof}

The following integral representation will be used at several places in this paper.

\begin{lemma}
\label{tildeA-e-calculation}
Let $q\in \R$ and $m\in\{1,\ldots,n-1\}$. Let $f=e^{-\varphi}$ be an upper semicontinuous log-concave function  on $\Rn$ such that $0<\int_{\Rn} f<\infty$, $o\in{\rm int}\,D_f$. %and $f(o)=\sup f=M>0$. 
If $\alpha_0,\beta_0>0$ and $\Upsilon:\R^n\to\R$ is Borel measurable such that $|\Upsilon(x)|\leq \alpha_0\|x\|+\beta_0$ holds for $x\in\R^n$, then
\begin{equation}
\label{tildeA-e-calculation-eq}
\int_{\Rn}\Upsilon\,d\widetilde{A}_{m,q}^{e}(f,\cdot)=\int_{{\rm G}(n,m)}\int_{\xi}
\left\|f|_\xi\right\|_1^{q-1}  \cdot \Upsilon(\nabla\varphi(x))\cdot f(x)\,dx\,d\xi,
\end{equation}
where  the integrands on both sides of \eqref{tildeA-e-calculation-eq} are integrable.
\end{lemma}

\begin{proof} 
We deduce from Proposition~\ref{tildeA-finiteness} that
\begin{equation}
\label{tildeA-e-calculation-finite}
\int_{\Rn}|\Upsilon|\,d\widetilde{A}_{m,q}^{e}(f,\cdot)\leq \int_{\Rn}(\alpha_0\|x\|+\beta_0)\,d\widetilde{A}_{m,q}^{e}(f,x)<\infty.
\end{equation}
Let $\Upsilon_+=\max\{\Upsilon,0\}$ and $\Upsilon_-=\min\{\Upsilon,0\}$ be the positive and negative parts, and let $C_k\subset{\rm int}\,D_f$ be a convex body for $k\in\N$ such that $C_k\subset C_{k+1}$ and $\bigcup_{k\in\N}C_k={\rm int}\,D_f$. 
As ${\rm supp}\,\widetilde{A}_{m,q}^{e}(f,\cdot)\subset{\rm cl}\,D_f$ (cf. Proposition~\ref{tildeA-finiteness}), we deduce from \eqref{tildeA-e-calculation-finite} that Lemma~\ref{tildeA-e-calculation} (including integrability) follows, if for any $k\in\N$, we have
\begin{align}
\label{tildeA-e-calculation-Psi+}    
\int_{\Rn}\Upsilon_+\mathbf{1}_{C_k}\,d\widetilde{A}_{m,q}^{e}(f,\cdot)=&\int_{{\rm G}(n,m)}\int_{\xi}
\Upsilon_+(\nabla\varphi)\cdot\mathbf{1}_{C_k}\left\|f|_\xi\right\|_1^{q-1}  \cdot f\,d\HH^m\,d\xi,\\
\label{tildeA-e-calculation-Psi-}    
\int_{\Rn}\Upsilon_-\mathbf{1}_{C_k}\,d\widetilde{A}_{m,q}^{e}(f,\cdot)=&\int_{{\rm G}(n,m)}\int_{\xi}\Upsilon_-(\nabla\varphi)\cdot\mathbf{1}_{C_k}
\left\|f|_\xi\right\|_1^{q-1}  \cdot  f\,d\HH^m\,d\xi.
\end{align}
We only prove \eqref{tildeA-e-calculation-Psi+} for $J_{k,+}=\int_{\Rn}\Upsilon_+\mathbf{1}_{C_k}\,d\widetilde{A}_{m,q}^{e}(f,\cdot)$, and the argument for \eqref{tildeA-e-calculation-Psi-} is analogous. 

For the finiteness estimates we need, we note that according to \eqref{log-concave-f-xi-A0A1}, there exist $A_1>A_0>1$ depending on $m$ and $f$ such that if $\xi\in{\rm G}(n,m)$, then $A_0\leq ({\rm R}_mf)\leq A_1$, and hence
\begin{equation}
\label{tildeA-e-calculation-A0A1}
({\rm R}_mf)^{q-1}\leq \max\left\{A_0^{q-1},A_1^{q-1}\right\}.
\end{equation}
Next, as $C_k\subset{\rm int}\,D_f$ is compact, there exists $M_k>0$ such that
$\|\nabla\varphi(x)\|\leq M_k$ holds for any $x\in C_k$, thus
\begin{equation}
\label{tildeA-e-calculation-CkphiUpsilon}
\Upsilon_+(\nabla\varphi(x))\leq \alpha_0M_k+\beta_0\mbox{ \ for }x\in C_k.
\end{equation}
Finally, \eqref{log-concave-f-exp-above} yields the existence of $\alpha>0$ and $\beta\in\R$ such that
$f(x)\leq e^{-(\alpha\|x\|+\beta)}$ for $x\in\R^n$, and hence for any $u\in S^{n-1}$, we have
\begin{equation}
\label{tildeA-e-calculation-intf-ray}
\int_0^\infty r^{m-1}f(ru)\,dr\leq
\int_0^\infty r^{m-1}e^{-(\alpha r+\beta)}\,dr<\infty.
\end{equation}
The finiteness of the integrals below follows from \eqref{tildeA-e-calculation-A0A1}, \eqref{tildeA-e-calculation-CkphiUpsilon} and \eqref{tildeA-e-calculation-intf-ray}. We deduce from \eqref{tildeAdef-Rn}, then applying polar coordinates, after that using the duality formula \eqref{Rm-dualRm-sphere}, and finally polar coordinates in $\xi\in{\rm G}(n,m)$ that
\begin{align*}
J_{k,+}=&\int_{\R^n} \Upsilon_+(\nabla\varphi(x)))\mathbf{1}_{C_k}(x)\|x\|^{m-n}f(x) \left(\widetilde{\mathcal{R}}_m^*({\rm R}_mf)^{q-1}\right)\left(\frac{x}{\|x\|}\right)\,dx,\\
=&\int_{S^{n-1}}\int_0^\infty r^{m-1}\Upsilon_+(\nabla\varphi(ru))\mathbf{1}_{C_k}(ru)f(ru) \left(\widetilde{\mathcal{R}}_m^*({\rm R}_mf)^{q-1}\right)(u)\,dr\,du\\
=&\int_{{\rm G}(n,m)}\left\|f|_\xi\right\|_1^{q-1}  \cdot\widetilde{\mathcal{R}}_m\int_0^\infty r^{m-1}\Upsilon_+(\nabla\varphi(ru))\mathbf{1}_{C_k}(ru)f(ru)\,dr\,d\xi\\
=&\int_{{\rm G}(n,m)}\int_{\xi}\Upsilon_+(\nabla\varphi)\cdot \mathbf{1}_{C_k}\left\|f|_\xi\right\|_1^{q-1}  \cdot  f\,d\HH^m\,d\xi,
\end{align*}
where we consider $\int_0^\infty r^{m-1}\Upsilon(\nabla\varphi(ru))\mathbf{1}_{C_k}(ru)f(ru)\,dr$ as a bounded function of $u\in S^{n-1}$. We conclude
\eqref{tildeA-e-calculation-Psi+},
and in turn Lemma~\ref{tildeA-e-calculation}.
\end{proof}

Finally, we show that the supports of the centro-sectional measures can't be too small.

\begin{lemma}
\label{Amq-pair-measures-support}
Let $q\in \R$,  and $m\in\{1,\ldots,n-1\}$.  If  $f$ is an upper semicontinuous  log-concave function  on $\Rn$ with $0<\int_{\Rn} f<\infty$ and $o\in{\rm int}\,D_f$, then 
no linear $(n-1)$-subspace contains the union of the supports of
 $\widetilde{A}_{m,q}^{e}(f,\cdot)$  and $\widetilde{A}_{m,q}^{s}(f,\cdot)$.
\end{lemma}
\begin{proof} 
The proof is indirect; namely, we suppose that there exists $v\in S^{n-1}$ such that ${\rm supp}\,\widetilde{A}_{m,q}^{e}(f,\cdot)\subset v^\bot$ and ${\rm supp}\,\widetilde{A}_{m,q}^{s}(f,\cdot)\subset v^\bot$, and seek a contradiction. Let $f=e^{-\varphi}$ for a convex function $\varphi$.

Let $\Omega'\subset{\rm int}\,D_f$ be the Borel set of the $\HH^n$ a.e. $x\in {\rm int}\,D_f$ where $\nabla \varphi(x)$ exists. In particular,  $\HH^1\big((y+\R v)\cap (D_f\backslash \Omega')\big)=0$ for $\HH^{n-1}$ a.e. $y\in ({\rm int}\,D_f)|v^\bot$. Since $\nabla \varphi$ is continuous on $\Omega'$ (cf. \eqref{nablaphi-continuous}), the density function involved in $\widetilde{A}_{m,q}^{e}(f,\cdot)$ is positive on $\Omega'\backslash\{o\}$ (cf. Definition~\ref{tildeAdef})  and ${\rm supp}\,\widetilde{A}_{m,q}^{e}(f,\cdot)\subset v^\bot$, we deduce that
$\nabla\varphi(x)\in v^\bot$ for each $x\in\Omega'$. It follows that $\varphi$, and in turn $f$, is constant along $(y+\R v)\cap ({\rm int}\,D_f)$ for $\HH^{n-1}$ a.e. $y\in ({\rm int}\,D_f)|v^\bot$, thus actually for all $y\in ({\rm int}\,D_f)|v^\bot$ as $\varphi$ and $f$ are continuous on ${\rm int}\,D_f$. As $f$ is upper continuous and the set $\{f\geq s\}$ is bounded for any $s>0$, we deduce that if $y\in ({\rm int}\,D_f)|v^\bot$, then
\begin{equation}
\label{Amq-pair-measures-support-fpos-boundary}
(y+\R v)\cap \partial D_f=\{y_+,y_-\}\mbox{ \ and \ }f(y_+),\,f(y_-)>0.
\end{equation}
In addition, if $X=(\partial' D_f)\cap(\R v+{\rm int}\,D_f)$ and $x\in X$, then \eqref{Amq-pair-measures-support-fpos-boundary} yields that $\HH^{n-1}(X)>0$, $u_{D_f}(x)\not\in v^\bot$ and $f(x)>0$,  and hence the density function involved in $\widetilde{A}_{m,q}^{s}(f,\cdot)$ is positive on $X$ (cf. Definition~\ref{tildeAdef}).  As $X\subset u_{D_f}^{-1}(S^{n-1}\backslash v^\bot)$, we deduce that
 $\widetilde{A}_{m,q}^{s}(f,S^{n-1}\backslash v^\bot)>0$, contradicting that ${\rm supp}\,\widetilde{A}_{m,q}^{s}(f,\cdot)\subset v^\bot$, and in turn proving Lemma~\ref{Amq-pair-measures-support}. 
\end{proof}

\section{Some basic estimates for log-concave functions}
\label{secBasic-Estimates}

The following  bounds on log-concave functions are consequences of the common sub-exponential behavior of the tail in the case of small variations, and are generalizations of \eqref{log-concave-f-exp-above} and \eqref{log-concave-f-xi-A0A1}. 

\begin{claim}
\label{exponential-upperbound}
Let $q\in\R$ and $m=1,\ldots,n-1$. Let $f=e^{-\varphi}$ be an upper semicontinuous log-concave function  on $\Rn$ such that $0<\int_{\Rn} f<\infty$ and $o\in{\rm int}\,D_f$,  and let $\sigma>0$ and $\zeta\in C_c(\R^n)$. Then there exist $t_0>0$, $A_1>A_0>0$, $\alpha>0$ and $\beta\in\R$ depending on $f$, $m$, $n$, $q$, $\sigma$ and $\zeta$ such that if  $\xi\in {\rm G}(n,m)$, $x\in\R^n$, $t\in[-t_0,t_0]$, $L\subset \sigma B^n$ is a compact convex set with $o\in L$, and $f_t=e^{-\varphi_t}$ where either $\varphi_t=(\varphi^*+th_L)^*$ or $\varphi_t=(\varphi^*+t\zeta)^*$, then  $o\in{\rm int}\,D_{f_t}$ and
\begin{align}
\label{exponential-upperbound-eq}
f_t(x)\leq &e^{-\alpha\|x\|-\beta},\\
\label{exponential-upperbound-xi-eq}
A_0\leq \left\|f_t|_\xi\right\|_1&\leq A_1,\\
\label{continuity-ftxi-eq0}
\lim_{t\to 0}\left\|f_t|_\xi\right\|_1=&\left\|f|_\xi\right\|_1.
\end{align}
\end{claim}
\begin{proof}
Let $p=f(o)>0$, let $M=\sup f$, and let $F_s=\left\{f\geq s\right\}$ for $s\in(0,M)$. 
As $0<\int_{\Rn} f<\infty$ and $o\in{\rm int}\,D_f$, there exist $\sigma>r>0$ such that
\begin{equation}
\label{exponential-upperbound-r}
2rB^n\subset F_{\frac{p}2}\mbox{ \ and \ }L\subset \sigma B^n, \mbox{ \ and let $t_0= \frac{r}{\sigma}$.}
\end{equation}

First, we verify that $o\in{\rm int}\,D_{f_t}$. 
If $\varphi_t=(\varphi^*+th_L)^*$ and $t\in[-t_0,t_0]$, then we deduce from $F_s\oplus (-t_0\sigma)\cdot B^n\subset \{f_t\geq s\}$ (cf. \eqref{log-sum-level} in Example~\ref{example-c1L}), \eqref{exponential-upperbound-r} and $\sigma|t|\leq r$ that
\begin{equation}
\label{exponential-upperbound-t0r}
rB^n \subset  \left\{f_t\geq \frac{p}2\right\}, \mbox{ \ and hence }rB^n \subset  \left\{f_t\geq s\right\}\mbox{ for }s\in\left(0,\frac{p}2\right],
\end{equation}
thus $o\in{\rm int}\,D_{f_t}$. On the other hand, if  $\varphi_t=(\varphi^*+t\cdot \zeta)^*$, then  there exists some $N>1$ such that $|\zeta(x)|\leq N$ for $x\in\R^n$. It follows that $\varphi(x)-|t|N\leq \varphi_t(x)\leq \varphi(x)+|t|N$ where $|t|N\leq \frac{rN}{\sigma}$, and hence (cf. \eqref{Legendre-shift} and \eqref{Legendre-monotone})
\begin{equation}
\label{exponential-upperbound-zeta-est}
e^{- \frac{rN}{\sigma}}f\leq e^{- N|t|}f\leq f_t\leq e^{ N|t|}f\leq e^{ \frac{rN}{\sigma}}f,
\end{equation} 
proving that ${\rm int}\,D_{f_t}={\rm int}\,D_{f}$, thus $o\in{\rm int}\,D_{f_t}$  in this case as well.

Next we prove \eqref{exponential-upperbound-eq}. If $\varphi_t=(\varphi^*+th_L)^*$, then
 as $h_L\geq 0$ by $o\in L$ and $L\subset \sigma B^n$, we have
$f_t\leq \tilde{f}:=e^{-(\varphi^*+t_0h_{\sigma B^n})^*}$ for $t\in[-t_0,t_0]$. We deduce from \eqref{log-concave-f-exp-above} applied to $\tilde{f}$ the existence of $\tilde{\alpha}>0$ and $\tilde{\beta}\in\R$ depending on $f,\sigma$ such that
$f_t(x)\leq \tilde{f}(x)\leq e^{-\tilde{\alpha}\|x\|-\tilde{\beta}}$ for any $x\in\R^n$ and $t\in[-t_0,t_0]$. 
In addition, we deduce from \eqref{exponential-upperbound-zeta-est}
that if $\varphi_t=(\varphi^*+t\cdot \zeta)^*$, then
 we may choose $\alpha=\tilde{\alpha}$ and $\beta=\tilde{\beta}-\frac{rN}{\sigma}$ in \eqref{exponential-upperbound-eq}.

For the upper bound in \eqref{exponential-upperbound-xi-eq}, we may choose $A_1=e^{-\beta}m\omega_m\int_0^\infty e^{-\alpha r}r^{m-1}\,dr$ by \eqref{exponential-upperbound-eq}. For the lower bound in \eqref{exponential-upperbound-xi-eq} if $\varphi_t=(\varphi^*+th_L)^*$, \eqref{exponential-upperbound-t0r} yields that
 $\left\|f_t|_\xi\right\|_1\geq \widetilde{A}_0$ holds for $\widetilde{A}_0=\frac{p}2\cdot r^m\omega_m$, $t\in[-t_0,t_0]$ and $\xi\in{\rm G}(n,m)$, and hence also $({\rm R}_mf)\geq \widetilde{A}_0$. In turn, if we take $\varphi_t=(\varphi^*+t\cdot \zeta)^*$ in the lower bound of \eqref{exponential-upperbound-xi-eq}, then $A_0=e^{- \frac{rN}{\sigma}}\widetilde{A}_0$ works by \eqref{exponential-upperbound-zeta-est}.

Finally, we apply Lebesgue's Dominant Convergence theorem to prove \eqref{continuity-ftxi-eq0}. Here the suitable pointwise upper bound is provided by \eqref{exponential-upperbound-eq}. If $\varphi_t=(\varphi^*+t\cdot \zeta)^*$, then the required pointwise limit is provided by \eqref{exponential-upperbound-zeta-est}. 

Therefore, we assume that $\varphi_t=(\varphi^*+th_L)^*$ holds in \eqref{continuity-ftxi-eq0}, and prove that
if $\lim_{k\to\infty}t_k=0$ for $t_k\neq 0$, then
\begin{equation}
\label{exponential-upperbound-L-est-pointwise}
\lim_{k\to\infty}f_{t_k}(x)=f(x) \mbox{ \ if }x\not\in\partial D_f.
\end{equation}
As $\sup f_{t_k}\leq M$, the level set $\{f_{t_k}\geq s\}=\emptyset$ for any $s>M$. On the other hand, if $s\in(0,M)$, then the level set $\{f_{t_k}\geq s\}=F_s\oplus t_k\cdot L$ tends to $F_s$, and hence $\varphi_{t_k}$ epiconverges to $\varphi$ by Definition~\ref{epiconvergence-def}. Thus   Lemma~\ref{epiconvergence-interior} yields \eqref{exponential-upperbound-L-est-pointwise}, and in turn, we conclude Claim~\ref{exponential-upperbound}.
\end{proof}

\begin{coro}
\label{Taylorqb}
Let $c>0$, $q\in\R$, $m=1,\ldots,n-1$. Let $f=e^{-\varphi}$ be an upper semicontinuous log-concave function  on $\Rn$ such that $0<\int_{\Rn} f<\infty$ and $o\in{\rm int}\,D_f$. Let $\sigma>0$ and $\zeta\in C_c(\R^n)$. Then there exist $t_0>0$, $A>0$, $\alpha>0$ and $\beta\in\R$ depending on $q$, $m$, $n$, $c$, $f$, $\sigma$ and $\zeta$ such that if  $\xi\in {\rm G}(n,m)$, $x\in\R^n$, $t\in[-t_0,t_0]$, $L\subset \sigma B^n$ is a compact convex set with $o\in L$, and $f_t=e^{-\varphi_t}$ where either $\varphi_t=(\varphi^*+t(h_L+\log c))^*$ or $\varphi_t=(\varphi^*+t\zeta)^*$, then there exists $b_{t,\xi}>0$ with the properties $o\in{\rm int}\,D_{f_t}$ and
\begin{align}
\label{Taylorqb-formula}
\left\|f_t|_\xi\right\|_1^q-\left\|f|_\xi\right\|_1^q=&q\,b_{t,\xi}^{q-1}\left(\left\|f_t|_\xi\right\|_1-\left\|f|_\xi\right\|_1\right)\mbox{ \ if }q\neq 0,\\
\label{Taylorqb-formula0}
\log\left\|f_t|_\xi\right\|_1-\log\left\|f|_\xi\right\|_1=&b_{t,\xi}^{-1}\left(\left\|f_t|_\xi\right\|_1-\left\|f|_\xi\right\|_1\right)\mbox{ \ if }q=0,\\
\label{Taylorqb-continuity}
b_{t,\xi}\mbox{ is a continuous }& \mbox{function of $\xi\in {\rm G}(n,m)$ for a fixed $t$,}\\
\label{Taylorqb-upper}
\left\|f_t|_\xi\right\|_1^{q-1}
\leq &A \mbox{ \ and \ }b_{t,\xi}^{q-1}\leq A,\\
\label{Taylorqb-limit}
\lim_{t\to 0}\left\|f_t|_\xi\right\|_1=\lim_{t\to 0}b_{t,\xi}=&\left\|f|_\xi\right\|_1.
\end{align}
\end{coro}
\begin{proof} As readily ${\rm dom}\,(\varphi^*+t(h_L+\log c))^*={\rm dom}\,(\varphi^*+th_L)^*$,  we deduce $o\in{\rm int}\,D_{f_t}$ from Claim~\ref{exponential-upperbound}.

The formulas \eqref{Taylorqb-formula} and \eqref{Taylorqb-formula0} defining $b_{t,\xi}$ follow from the Taylor formula where $b_{t,\xi}$ lies between 
$\left\|f_t|_\xi\right\|_1$ and $\left\|f|_\xi\right\|_1$, and in turn \eqref{Taylorqb-continuity} follows from Lemma~\ref{fxi-representations}. Next, Example~\ref{example-c1L} and Claim~\ref{exponential-upperbound} yield that
$$
\left\|f_t|_\xi\right\|_1^{q-1}\leq \max\left\{1,c^{t(q-1)}\right\}\cdot \max\left\{A_0^{q-1},A_1^{q-1}\right\}\leq 2\max\left\{A_0^{q-1},A_1^{q-1}\right\}
$$
if $t_0>0$ is small enough,
which estimate verifies the existence of $A$ in \eqref{Taylorqb-upper} as $b_{t,\xi}$ lies between 
$\left\|f_t|_\xi\right\|_1$ and $\left\|f|_\xi\right\|_1$. Finally, \eqref{Taylorqb-limit} follows again from  Example~\ref{example-c1L} and Claim~\ref{exponential-upperbound}.
\end{proof}

As a counterpart of \eqref{exponential-upperbound-eq} in Claim~\ref{exponential-upperbound}, we have the following.

\begin{claim}
\label{exponential-lowerbound}
Let $f=e^{-\varphi}$ be a log-concave function  on $\Rn$ such that $0<\int_{\Rn} f<\infty$ and 
$f(x_0)=\sup f=M$ for some $x_0\in{\rm int}\,D_f$. If $x_0+rB^n\subset {\rm int}\,\{f\geq s_0\}$ for some $r>0$ and $s_0\in(0,M)$, then there exists $\gamma>0$ depending on $M$, $r$ and $s_0$ such that
for any $s\in[s_0,M)$, we have
\begin{equation}
\label{exponential-lowerbound-radius}
 x_0+\gamma(M-s)\,B^n\subset \{f\geq s\}.
\end{equation}
\end{claim}
\begin{proof}
For $x\in x_0+rB^n$, the convexity of $\varphi$ and $\varphi(x_0)=-\log M$ yield that $\varphi(x)\leq -\log M+\frac{-\log s_0+\log M}{r}\,\|x-x_0\|$, implying $f(x)\geq Me^{-\gamma_0\|x-x_0\|}$ where $\gamma_0=\frac{\log(M/s_0)}r$. In turn, we deduce \eqref{exponential-lowerbound-radius} from  $e^{-\gamma t}> 1-\gamma t$ for $t>0$.
\end{proof}

\section{An estimate for moments on a difference set}
\label{secpower-integral}

This section is dedicated to estimate the variation of moments. 

\begin{lemma}
\label{tildeV-theta-parallelset}
 For $\theta\in[1,n]$ and $R>\sigma>e$, there exists $\gamma>0$ depending on $\theta,R,\sigma,n$ such that if $t\in \left[-\frac{1}{2R},\frac{1}{2R}\right]$ and $F\subset RB^n$ is a convex body and $L\subset \sigma B^n$ is a compact convex set with $o\in L$, then (cf. \eqref{FoplustL})
\begin{align}
\label{tildeV-theta-large-parallelset-eq}
\int_{F\Delta (F\oplus t\cdot L)}\|x\|^{\theta-n}\,dx\leq &\gamma\cdot |t|&&\mbox{if } \theta>1 ;\\
\label{tildeV-theta1-parallelset-eq}
\int_{F\Delta (F\oplus t\cdot L)}\|x\|^{1-n}\,dx\leq &\gamma\cdot |t|\cdot\log \frac1{|t|} &&\mbox{if $\theta=1$ and $t\neq 0$.}\\
\end{align}
In addition, if $\theta=1$, $t\neq 0$ and $\aleph B^n\subset F$ with $2|t|\leq \aleph\leq \frac{1}{R}$, then
\begin{equation}
\label{tildeV-theta1aleph-parallelset-eq}
\int_{F\Delta (F\oplus t\cdot L)}\|x\|^{1-n}\,dx\leq \gamma\cdot |t|\cdot\log \frac1{\aleph}.
\end{equation}
\end{lemma}
\begin{proof}
Since $F\subset F\oplus tL\subset F\oplus t\sigma B^n$ if $t>0$, and 
$F\oplus t\sigma B^n\subset F\oplus tL \subset F$ if $t<0$, we may assume that $L=\sigma B^n$ and $t\neq 0$. We observe that if $x\in \R^n\backslash \partial F$ ,  and $z\in\partial F$ is a closest point of $\partial F$ to $x$, then
\begin{align}
\label{tildeV-theta-parallelset-x-z}
x-z\mbox{ \ is an exterior normal to $F$ at $z$ }&\mbox{if \ }x\not\in F;\\
\label{tildeV-theta-parallelset-z-x}
z-x\mbox{ \ is an exterior normal to $F$ at $z$ }&\mbox{if \ }x\in {\rm int}\,F.
\end{align}
Let $e_1,\ldots,e_n$ be an orthonormal basis of $\R^n$, and let $e_{i+n}=-e_i$ for $i=1,\ldots, n$. For $i=1,\ldots, 2n$, we define the closed set $\Xi_i\subset {\rm cl}\left(F\Delta (F\oplus t\cdot L)\right)$ as follows. We have $x\in \Xi_i$, if
\begin{itemize}
\item either $x\in \partial F$, and there exists an exterior normal $u\in S^{n-1}$ at $x$ to $F$ with $u\cdot e_i\geq \frac1{\sqrt{n}}$,
\item or $t>0$, $x\in (F\oplus t\cdot L)\backslash F$, and $(x-z)\cdot e_i\geq \frac1{\sqrt{n}}\,\|x-z\|$ holds for the closest point $z\in\partial F$ of $F$ to $x$,
\item or $t<0$, $x\in ({\rm int}\,F)\backslash (F\oplus t\cdot L)$, and $(z-x)\cdot e_i\geq \frac1{\sqrt{n}}\,\|z-x\|$ holds for some closest point $z\in\partial F$ of $ F$ to $x$.
\end{itemize}
Since for any $v\in \R^n$, there exists an $e_i$ such that $v\cdot e_i\geq \frac1{\sqrt{n}}\,\|v\|$, it follows from \eqref{tildeV-theta-parallelset-x-z} and \eqref{tildeV-theta-parallelset-z-x} that 
\begin{equation}
\label{tildeV-theta-parallelset-F-cupXii}
F\Delta (F\oplus t\cdot L)\subset \bigcup_{i=1}^{2n}\Xi_i.
\end{equation}

We claim for any $i=1,\ldots, 2n$ and $y\in \Xi_i|e_i^\bot$, we have
\begin{equation}
\label{tildeV-theta-parallelset-Xii-section}
\mathcal{H}^1((y+\R e_i)\cap \Xi_i)\leq
4n \sigma\cdot |t|.
\end{equation}
Let $t_0=\min\{t\in \R:y+t e_i\in \Xi_i\}$ and $t_1=\max\{\tau\in \R:y+\tau e_i\in \Xi_i\}$, and hence $\mathcal{H}^1((y+\R e_i)\cap \Xi_i)\leq t_1-t_0$. For $j=0,1$, let $x_j=y+t_j e_i$, and let $z_j\in\partial F$ be a closest point of $F$ to $x_j$  where 
$x_j-z_j=\pm \|x_j-z_j\|u_j$ for an exterior unit normal $u_j$ to $F$ at $z_j$ with $u_j\cdot e_i\geq \frac1{\sqrt{n}}$.
We write $z_j$ in the form  $z_j=y+w_j+s_je_i$ for $j=0,1$ where $w_j\in e_i^\bot$ and $s_j\in \R$; therefore,
$\sigma^2|t|^2\geq \|x_j-z_j\|^2=\|w_j\|^2+|t_j-s_j|^2$ implies that  $\|w_j\|\leq \sigma|t|$ and $|t_j-s_j|\leq \sigma|t|$.
To estimate $|s_1-s_0|$, if $s_1\geq s_0$, then $(z_1-z_0)\cdot u_0\leq 0$ and $e_i\cdot u_0\geq \frac1{\sqrt{n}}$ yield that 
$$
0\geq (z_1-z_0)\cdot u_0=(w_1-w_0+(s_1-s_0)e_i)\cdot u_0\geq -2\sigma|t|+\frac{s_1-s_0}{\sqrt{n}}, 
$$
and hence $s_1-s_0\leq 2\sqrt{n}\sigma |t|$. On the other hand, if $s_1\leq s_0$, then $(z_0-z_1)\cdot u_1\leq 0$ and $e_i\cdot u_1\geq \frac1{\sqrt{n}}$ yield that 
$$
0\geq (z_0-z_1)\cdot u_1=(w_0-w_1+(s_0-s_1)e_i)\cdot u_1\geq -2\sigma|t|+\frac{s_0-s_1}{\sqrt{n}}, 
$$
and hence $s_0-s_1\leq 2\sqrt{n}\sigma|t|$; therefore, we have $|s_1-s_0|\leq 2\sqrt{n}\sigma|t|$ independently whether $s_1\geq s_0$ or $s_1\leq s_0$. In summary, we deduce using $(z_1-z_0)\cdot e_i=s_1-s_0$ that
$$
t_1-t_0=(x_1-x_0)\cdot e_i=(x_1-z_1)\cdot e_i-(x_0-z_0)\cdot e_i+(z_1-z_0)\cdot e_i\leq 2\sigma|t|+2\sqrt{n}\sigma|t|,
$$
thus $\mathcal{H}^1((y+\R e_i)\cap \Xi_i)\leq t_1-t_0$ implies the claim \eqref{tildeV-theta-parallelset-Xii-section}.\\

Therefore, let $i\in\{1,\ldots,2n\}$ and $t\in \left[-\frac{1}{2R},\frac{1}{2R}\right]$, and hence 
\begin{equation}
\label{FtL-2RB}
\Xi_i\cup F\cup (F\oplus t\cdot L)\subset 2R B^n.
\end{equation}
As $\theta\leq n$, if $x\not\in\R e_i$ and $y=x|e_i^\bot$, then
\begin{equation}
\label{tildeV-theta-parallelset-xyeibot}
\|x\|^{\theta-n}\leq \|y\|^{\theta-n},
\end{equation}
and as the function $\tau\mapsto \tau^{\theta-n}$ is decreasing for $\tau>0$, it follows from the claim \eqref{tildeV-theta-parallelset-Xii-section} that if $y\in e_i^\bot$, then
\begin{equation}
\label{tildeV-theta-parallelset-xyeibot-sym}
\int_{(y+\R e_i)\cap \Xi_i}\|x\|^{\theta-n}\,d\HH^1(x)\leq \int_{y+2n\sigma|t|{\rm conv}\{-e_ie_i\}}\|x\|^{\theta-n}\,d\HH^1(x).
\end{equation}

We divide the rest of the argument into cases depending on the value of $\theta$.
If $1<\theta\leq n$, then
 the claim \eqref{tildeV-theta-parallelset-Xii-section}, the estimate \eqref{tildeV-theta-parallelset-xyeibot}, Fubini's theorem, and the use of polar coordinates in $e_i^\bot\cap (2RB^n)$ yield that
\begin{align*}
\int_{\Xi_i}\|x\|^{\theta-n}\,dx\leq &4n\sigma \cdot |t| \int_{e_i^\bot\cap 2R B^n}\|y\|^{\theta-n}\,d\mathcal{H}^{n-1}(y)\\
=&
4n\sigma\cdot  |t|\cdot (n-1)\omega_{n-1}\int_0^{2R}r^{\theta-n}r^{n-2}\,dr=\gamma'\cdot|t|
\end{align*}
for $\gamma'=\frac{4n(n-1)\omega_{n-1}(2R)^{\theta-1}\sigma}{\theta-1}$, proving \eqref{tildeV-theta-large-parallelset-eq} by \eqref{tildeV-theta-parallelset-F-cupXii}.

Now, let $\theta=1$ and $t\neq 0$. To prove \eqref{tildeV-theta1-parallelset-eq}, if $i\in\{1,\ldots,2n\}$, then let $\Xi_{i,0}=\{x\in \Xi_i:\|x|e_i^\bot\|\leq 2n\sigma|t|\}$ and $\Xi_{i,+}=\Xi_i\backslash \Xi_{i,0}$. We deduce from Fubini's theorem and \eqref{tildeV-theta-parallelset-xyeibot-sym} that
\begin{align}
\nonumber
\int_{\Xi_{i,0}}\|x\|^{1-n}\,dx\leq &\int_{2n\sigma|t|(B^n\cap e_i^\bot)+2n\sigma|t|{\rm conv}\{-e_i,e_i\}}\|x\|^{1-n}\,dx\\
\label{tildeV-theta1-parallelset-eq0}
\leq & 
\int_{4n\sigma|t|\,B^n}\|x\|^{1-n}\,dx
=4n^2\sigma\omega_n|t|.
\end{align}
On the other hand, if $x\in \Xi_{i,+}$, then simply using that $\|x\|^{1-n}\leq \|y\|^{1-n}$ holds for $y=x|e_i^\bot$, we deduce again from claim \eqref{tildeV-theta-parallelset-Xii-section} and \eqref{FtL-2RB} that
\begin{align}
\nonumber
\int_{\Xi_{i,+}}\|x\|^{1-n}\,dx\leq &4n\sigma \cdot |t| \int_{e_i^\bot\cap [(nR B^n)\backslash (2n\sigma |t|B^n)]}\|y\|^{1-n}\,d\mathcal{H}^{n-1}(y)\\
\nonumber
=&4n\sigma(n-1)\omega_{n-1}\cdot |t|\big(\log(nR)-\log(2n\sigma|t|)\big)\\
\label{tildeV-theta1-parallelset-eq+}
\leq& 8n^2\omega_{n-1}\sigma\cdot |t|\log\frac1{|t|},
\end{align}
as $\log(nR)-\log(2n\sigma|t|)=\log\frac{R}{2\sigma}+\log \frac1{|t|}\leq 2\log \frac1{|t|}$. The estimate  \eqref{tildeV-theta1-parallelset-eq+} combined with  \eqref{tildeV-theta-parallelset-F-cupXii} and \eqref{tildeV-theta1-parallelset-eq0} yields \eqref{tildeV-theta1-parallelset-eq}.

Still, when $\theta=1$ and $t\neq 0$, our final goal is to prove \eqref{tildeV-theta1aleph-parallelset-eq}, and hence we have $\aleph B^n\subset F$. If $i=1,\ldots,2n$, then now let $\Xi_{i,0}=\{x\in \Xi_i:\|x|e_i^\bot\|\leq \frac{\aleph}2\}$, and $\Xi_{i,+}=\Xi_i\backslash \Xi_{i,0}$. 
As $\frac{\aleph}2\,B^n\subset F\oplus t\cdot L$ follows from $\aleph\geq 2|t|$, 
\begin{equation}
\label{being-aleph-away}
\mbox{$\|x\|^{\theta-n}\leq (\frac{\aleph}2)^{\theta-n}$ holds for
$x\in \Xi_{i,0}$.}
\end{equation}
Now, the  claim \eqref{tildeV-theta-parallelset-Xii-section} and \eqref{being-aleph-away} imply that
\begin{equation}
\label{tildeV-theta1aleph-parallelset-eq0}
\int_{\Xi_{i,0}}\|x\|^{1-n}\,dx\leq 4n|t|\int_{\frac{\aleph}2(B^n\cap e_i^\bot)}\left(\frac{\aleph}2\right)^{1-n}\,dx=
4n\omega_{n-1}|t|.
\end{equation}
On the other hand, if $x\in \Xi_{i,+}$, then simply using that $\|x\|^{1-n}\leq \|y\|^{1-n}$ holds for $y=x|e_i^\bot$, we deduce again from the claim \eqref{tildeV-theta-parallelset-Xii-section} and  $\aleph\leq \frac1{R}$ that
\begin{align*}
\int_{\Xi_{i,+}}\|x\|^{1-n}\,dx\leq &4n \cdot |t| \int_{e_i^\bot\cap (2R B^n\backslash \frac{\aleph}2B^n)}\|y\|^{1-n}\,d\mathcal{H}^{n-1}(y)\\
=&4n(n-1)\omega_{n-1}\cdot |t|\left(\log(2R)+\log\frac2{\aleph}\right)\\
\leq &16n^2\omega_{n-1}\cdot |t|\cdot \log\frac1{\aleph}
\end{align*}
as $e\leq R\leq \frac1{\aleph}$,
which estimate combined with \eqref{tildeV-theta-parallelset-F-cupXii} and  \eqref{tildeV-theta1aleph-parallelset-eq0}  yields \eqref{tildeV-theta1aleph-parallelset-eq},
completing the proof of Lemma~\ref{tildeV-theta-parallelset}.
\end{proof}

\section{Variational formulas as $t\to 0$ }
\label{secVariational-formulas}

Before proving the variational formula with respect to a compact convex set $L$, we recall a variational formula from \cite{HLYZ16} for a convex body $F\subset\R^n$ with $o\in{\rm int}\,F$. If the function $g: S^{n-1}\to\R$ is continuous and $|t|$ is sufficiently small, then let $h_t(v)=h_F(v)+tg(v)$ for $v\in S^{n-1}$, and let
$$
F_{(t)}=\{x\in \R^n: x\cdot v\leq h_t(v) \  \forall v\in S^{n-1}\}
$$ 
that is a convex body with $o\in{\rm int}\,F_t$.
Then for $\HH^{n-1}$ a.e. $u\in S^{n-1}$ and $x=\varrho_F(u)u$ (and hence $\|x\|=\varrho_{F}(u)$), \cite{HLYZ16} proves that
\begin{equation}
\label{HLYZ-radial-der}
\left.\frac{d}{dt}\varrho_{F_{(t)}}(u)\right|_{t=0}=\frac{g(\alpha_F(u))\cdot \varrho_F(u)}{h_F(\alpha_F(u))}=\frac{g(u_F(x))\cdot \|x\|}{h_F(u_F(x))}.
\end{equation}

\begin{prop}
\label{variation-phi+thL} 
Let $q\in\R$ and $m\in\{1,\ldots,n-1\}$. Let $f=e^{-\varphi}$ be an upper semicontinuous log-concave function  on $\Rn$ such that $0<\int_{\Rn} f<\infty$, $o\in{\rm int}\,D_f$ and $f(o)=\sup f>0$. Let $L\subset \R^n$ be a compact convex set with $o\in L$. Then $f_t=e^{-\varphi_t}$ for $\varphi_t=(\varphi^*+th_L)^*$ and $t\in\R$ satisfies  that 

if $q\neq 0$, then
\begin{align}
\label{variation-phi+thL-eq}
\lim_{t\to 0}\frac{\widetilde{\Psi}_{m,q}(f_t)-\widetilde{\Psi}_{m,q}(f)}t
=&q\int_{\Rn}h_L\,d\widetilde{A}_{m,q}^{e}(f,\cdot)+q\int_{S^{n-1}}h_L\,d\widetilde{A}_{m,q}^{s}(f,\cdot);\\
\label{variation-phi+thL-eq2}
\lim_{t\to 0}\frac{\widetilde{\Psi}_{m,0}(f_t)-\widetilde{\Psi}_{m,0}(f)}t
=&\int_{\Rn}h_L\,d\widetilde{A}_{m,0}^{e}(f,\cdot)+\int_{S^{n-1}}h_L\,d\widetilde{A}_{m,0}^{s}(f,\cdot),
\end{align}
where both integrals occurring in \eqref{variation-phi+thL-eq} and in \eqref{variation-phi+thL-eq2} are finite.
\end{prop}
\begin{proof} First, let $q\neq 0$.
During the argument, we abbreviate $d\nu_{m}(\xi)$ to $d\xi$ when integrating on ${\rm G}(n,m)$. Let
$$
L\subset \sigma B^n \mbox{ \  for }\sigma>e.
$$ 
The finiteness of the two integrals occurring in \eqref{variation-phi+thL-eq} and in \eqref{variation-phi+thL-eq2} follows from $h_L(x)\leq  \sigma\|x\|$ for $x\in \R^n$ and Proposition~\ref{tildeA-finiteness}.

Let $M=\sup f$ and  $F_s=\{f\geq s\}$ for $s\in(0,M]$. We frequently use that for any $s\in(0,M)$, we have (cf. \eqref{log-sum-level})
\begin{equation}
\label{FsL-fts}
\begin{array}{rcl}
F_s\subset \{f_t\geq s\}=F_s\oplus t\cdot L&\mbox{ if }&t>0,\\
F_s\supset \{f_t\geq s\}=F_s\oplus t\cdot L&\mbox{ if }&t<0.
\end{array}
\end{equation}
There exist an $R_0>\sigma$ and $r_0>0$ such that
\begin{equation}
\label{variation-phi-thL-r0-R0}
r_0 B^n\subset F_{M/2}\subset R_0B^n, \mbox { \ and hence } F_s\subset R_0B^n\mbox{ if }\frac{M}2\leq s<M.
\end{equation}
We deduce from \eqref{exponential-lowerbound-radius} in Claim~\ref{exponential-lowerbound} the existence of $\gamma_0\in(0,1)$ (depending on $f$ and $r_0$) such that 
\begin{equation}
\label{variation-phi-thL-gamma0-M-s}
 \gamma_0(M-s)\,B^n\subset F_s\mbox{ \ if }s\in\left[\frac{M}2,M\right).
\end{equation}
 
We choose a $t_0\in(0,\frac1{2R_0})$ depending on $f$ and  $L$ such that Corollary~\ref{Taylorqb} holds for $t\in(-t_0,t_0)$ when $c=1$ and $q$, $m$, $n$, $f$, $\sigma$ are the same. 
According to Corollary~\ref{Taylorqb}, there exist $A>0$ depending on $f$ and $\sigma$ such that if $t\in(-t_0,t_0)$ and $\xi\in{\rm G}(n,m)$, then
for certain $b_{t,\xi}>0$, we have
\begin{align}
\label{variation-phi-thL-ftxi-q-1}
\left\|f_t|_\xi\right\|_1^{q-1}\leq &A,\\ 
\label{variation-phi-thL-fxi-tautxi1}
\left\|f_t|_\xi\right\|_1^q-\left\|f|_\xi\right\|_1^q=&q\cdot b_{t,\xi}^{q-1}\cdot \left(\left\|f_t|_\xi\right\|_1-\left\|f|_\xi\right\|_1\right),\\
\label{variation-phi-thL-fxi-tautxi2}
b_{t,\xi}^{q-1}\leq &A,\\
\label{variation-phi-thL-fxi-tautxi-limit}
\lim_{t\to 0}b_{t,\xi}=&\left\|f|_\xi\right\|_1. 
\end{align}
If $\theta\in(0,M]$ and $0<|t|<t_0$, then we define
\begin{align*}
I(\theta,t)=&\frac{q}m\int_0^\theta\int_{{\rm G}(n,m)}\int_{S^{n-1}\cap\xi}b_{t,\xi}^{q-1}\times\\
&\times \frac{\left({\rm sign}\,\varrho_{\{f_t> s\}}(u)\right)\left|\varrho_{\{f_t> s\}}(u)\right|^m-\varrho_{F_s}(u)^m}{t}\,d\HH^{m-1}(u)\,d\xi\,ds, \\
J(\theta)=&q\int_0^\theta\int_{{\rm G}(n,m)}\int_{S^{n-1}\cap\xi}\left\|f|_\xi\right\|_1^{q-1} \varrho_{F_s}(u)^m
\frac{h_L(\alpha_{F_s}(u))}{h_{F_s}(\alpha_{F_s}(u))}\,d\HH^{m-1}(u)\,d\xi\,ds.
\end{align*}
We deduce from \eqref{FsL-fts} that if $u\in S^{n-1}$, $0<|t|<t_0$ and $s\in(0,M)$, then
\begin{equation}
\label{FsL-fts-upsilon-pos}
\Upsilon(u,t,s):=\frac{\left({\rm sign}\,\varrho_{\{f_t> s\}}(u)\right)\left|\varrho_{\{f_t> s\}}(u)\right|^m-\varrho_{F_s}(u)^m}{t}>0.
\end{equation}
For fixed $t$ with $0<|t|<t_0$, it follows that $I(\theta,t)$ is an increasing function of $\theta\in(0,M]$; moreover,
 \eqref{variation-phi-thL-fxi-tautxi1} and later \eqref{fxi-representations-Rmlevelset-eq} in Lemma~\ref{fxi-representations} yield that 
\begin{align}
\nonumber
\frac{\widetilde{\Psi}_{m,q}(f_t)-\widetilde{\Psi}_{m,q}(f)}t=&\int_{{\rm G}(n,m)}\frac{\|f_t|_{\xi}\|_1^q-\|f|_{\xi}\|_1^q}t\,d\xi\\
\nonumber
=&
q \int_{{\rm G}(n,m)}b_{t,\xi}^{q-1}\cdot \frac{\left\|f_t|_\xi\right\|_1-\left\|f|_\xi\right\|_1}t\,d\xi\\
\nonumber
=&\frac{q}m\int_{{\rm G}(n,m)}b_{t,\xi}^{q-1} \int_0^M\int_{S^{n-1}\cap\xi} \Upsilon(u,t,s)\,d\HH^{m-1}(u)\,ds\,d\xi\\
\label{variation-phi-thL-tildePsimq-Ithetat}
 =&I(M,t).
\end{align}
For the finiteness of $J(\theta)$, and in turn to show $\lim_{\theta\to M^{-}}(J(M)-J(\theta))=0$, we claim that
\begin{equation}
\label{JM-finite}
J(M)<\infty.
\end{equation}
We deduce from \eqref{variation-phi-thL-ftxi-q-1} and $L\subset\sigma B^n$, then from the duality formula \eqref{Rm-dualRm-sphere} and setting $\gamma^*=\frac{|q|A\sigma m\omega_m}{n\omega_n}$, after that from the formula \eqref{HLYZ-cone-volume0}, then from the coarea formula \eqref{coarea-BV-eq} where $\nabla f=-f\nabla\varphi$ that
\begin{align*} 
\nonumber
J(M)\leq &|q|A\sigma\int_0^M\int_{{\rm G}(n,m)}\widetilde{\mathcal{R}}_m 
\frac{\varrho_{F_s}(u)^m}{h_{F_s}(\alpha_{F_s}(u))}\,d\xi\,ds\\
\nonumber
=& \gamma^*\int_0^M\int_{S^{n-1}}
\frac{\varrho_{F_s}(u)^m}{h_{F_s}(\alpha_{F_s}(u))}\,du\,ds
= \gamma^*\int_0^M\int_{\partial F_s} \|x\|^{m-n}
\,dx\,ds \\
= &\gamma^*\int_{D_f} \|x\|^{m-n} \|\nabla f(x)\|
\,dx+\gamma^*\int_{\partial D_f}\|x\|^{m-n} f(x)\,dx.
\end{align*}
Now, there exists some $r_*>0$ such that $r_*B^n\subset {\rm int}\,D_f$, and hence one finds $M_*>0$ such that $\|\nabla f(x)\|=f(x)\|\nabla\varphi(x)\|\leq M_*$ for any $x\in r_*B^n$ where $\nabla\varphi(x)$ exists. It follows that
\begin{align*}
\int_{D_f} \|x\|^{m-n} \|\nabla f(x)\|\,dx\leq &M_*\int_{r_*B^n} \|x\|^{m-n} \,dx+
r_*^{m-n}\int_{\R^n}\|\nabla f(x)\|\,dx<\infty,\\
\int_{\partial D_f}\|x\|^{m-n} f(x)\,d\HH^{n-1}(x)\leq &r_*^{m-n}\int_{\partial D_f}f(x)\,d\HH^{n-1}(x)<\infty
\end{align*}
by $1\leq m\leq n-1$ and \eqref{Dario-Boaz-Liran-eq} after Lemma~\ref{coarea-BV}, proving the claim \eqref{JM-finite}.

We divide the rest of the proof of Proposition~\ref{variation-phi+thL} into two main steps.\\

\noindent{\bf Step 1.} $\lim_{t\to 0}\frac{\widetilde{\Psi}_{m,q}(f_t)-\widetilde{\Psi}_{m,q}(f)}t=J(M)$.

The core of the argument in Step~1  is the claim that for any small $\varepsilon\in(0,1)$, there exist 
$\eta_\varepsilon\in(0,\frac{M}2]$ such that also $\eta_\varepsilon\leq \varepsilon$, and for $t_\varepsilon=\min\{t_0,\frac{\gamma_0\eta_\varepsilon}{4m\sigma}\}$ (cf. \eqref{variation-phi-thL-gamma0-M-s}), we have  
\begin{align}
\label{variation-phi-thL-Ithetat-eta}
|I(M,t)-I(M-\eta_\varepsilon,t)|\leq &\varepsilon&&\mbox{if }0<|t|<t_\varepsilon,\\
\label{variation-phi-thL-Ithetat-limit}
\lim_{t\to 0}I(M-\eta_\varepsilon,t)=&J(M-\eta_\varepsilon).&&
\end{align}

To prove \eqref{variation-phi-thL-Ithetat-eta}, let $\eta\in(0,\frac{M}2]$ such that $\eta\leq \frac1{\gamma_0R_0}$, and let $0<|t|<\tilde{t}_\eta$ where $\tilde{t}_\eta=\min\{t_0,\frac{\gamma_0\eta}{4m\sigma}\}$. 
Using the notation as in \eqref{FsL-fts-upsilon-pos} and $\tilde{\gamma}=\frac{m\omega_m |q|A}{n\omega_n }$, it follows from applying \eqref{variation-phi-thL-fxi-tautxi2}, then \eqref{rho-tom-xi-inner-int}, and after  that \eqref{FsL-fts} that 
\begin{align}
\nonumber
|I(M,t)-I(M-\eta,t)|\leq &
\frac{|q|A}m\int_{M-\eta}^M\int_{{\rm G}(n,m)}\int_{S^{n-1}\cap\xi}\Upsilon(u,t,s)\,d\HH^{m-1}(u)\,d\xi\,ds\\
\nonumber
=&\tilde{\gamma}\int_{M-\eta}^M\frac{\int_{\{f_t\geq s\}}\|x\|^{m-n}\,dx-\int_{F_s}\|x\|^{m-n}\,dx}t\,ds\\
\label{variation-phi-thL-Ithetat-basicest}
=&\tilde{\gamma}\int_{M-\eta}^M\frac1{|t|}\int_{F_s\Delta(F_s\oplus t\cdot L)}\|x\|^{m-n}\,dx\,ds.
\end{align}
Let $\gamma>0$ be the constant of  Lemma~\ref{tildeV-theta-parallelset} for $\sigma>e$ and $R=R_0>\sigma$,
and we observe that if $0<|t|<\tilde{t}_\eta$ and $s\in[M-\eta,M)$, then
\begin{equation}
\label{variation-phi-thL-Ithetat-basicest-m1-pcond0}
F_s\subset R_0B^n \mbox{ \ and \ } |t|\leq \frac{\gamma_0\eta}{4m\sigma}<\frac1{2R_0}<\frac1e\mbox{ \ and \ }R_0>\sigma>e. 
\end{equation}
If $m\geq 2$, then \eqref{tildeV-theta-large-parallelset-eq} in Lemma~\ref{tildeV-theta-parallelset}, \eqref{variation-phi-thL-Ithetat-basicest-m1-pcond0} and \eqref{variation-phi-thL-Ithetat-basicest} yield that
$$
|I(M,t)-I(M-\eta,t)|\leq \tilde{\gamma}\gamma \eta;
$$
therefore, we choose a $\eta_\varepsilon>0$ for \eqref{variation-phi-thL-Ithetat-eta} such that $\eta_\varepsilon\leq \min\{ \frac{M}2,\frac1{\gamma_0R_0},\varepsilon\}$ and $\tilde{\gamma} \gamma \eta_\varepsilon\leq \varepsilon$.

If $m=1$, then we rewrite  \eqref{variation-phi-thL-Ithetat-basicest} in the form
\begin{align}
\label{variation-phi-thL-Ithetat-basicest-m1-t}
|I(M,t)-I(M-\eta,t)|\leq& \tilde{\gamma}\int_{0}^{\frac{4}{\gamma_0}\,|t|}\frac1{|t|}\int_{F_{M-p}\Delta(F_{M-p}\oplus t\cdot L)}\|x\|^{1-n}\,dx\,dp+\\
\label{variation-phi-thL-Ithetat-basicest-m1-eta}
&+ \tilde{\gamma}\int_{\frac{4}{\gamma_0}\,|t|}^\eta\frac1{|t|}\int_{F_{M-p}\Delta(F_{M-p}\oplus t\cdot L)}\|x\|^{1-n}\,dx\,dp,
\end{align}
where $\frac{4}{\gamma_0}\,|t|<\eta$ by $|t|\leq \frac{\gamma_0\eta}{4m\sigma}$,  and if $\frac{4}{\gamma_0}\,|t|\leq p\leq \eta$, then (cf. \eqref{variation-phi-thL-r0-R0} and \eqref{variation-phi-thL-gamma0-M-s})
\begin{equation}
\label{variation-phi-thL-Ithetat-basicest-m1-pcond}
\gamma_0p\,B^n\subset  F_{M-p}\subset R_0B^n\mbox{ \ and \ } 2|t|<\gamma_0p\leq \frac1{R_0}<\frac1e.
\end{equation}
Applying \eqref{tildeV-theta1-parallelset-eq} in Lemma~\ref{tildeV-theta-parallelset} to the right hand side of \eqref{variation-phi-thL-Ithetat-basicest-m1-t} (cf. \eqref{variation-phi-thL-Ithetat-basicest-m1-pcond0}), and
\eqref{tildeV-theta1aleph-parallelset-eq} in Lemma~\ref{tildeV-theta-parallelset} to the right hand side of \eqref{variation-phi-thL-Ithetat-basicest-m1-eta} (cf. \eqref{variation-phi-thL-Ithetat-basicest-m1-pcond}), we deduce using Claim~\ref{tlogt-monotonicity}, $(p+p\log\frac1{\gamma_0 p})'=\log\frac1{\gamma_0 p}$ and  $|t|<\gamma_0\eta<\frac1e$ that
\begin{align*}
|I(M,t)-I(M-\eta,t)|\leq& \tilde{\gamma} \cdot \frac{4|t|}{\gamma_0}\cdot \gamma \log\frac1{|t|}+
\tilde{\gamma}\int_{0}^\eta \gamma\log\frac1{\gamma_0 p}\,dp\\
= &\tilde{\gamma}\gamma\left(\frac{4|t|}{\gamma_0}\cdot \log\frac1{|t|}+\eta+
\eta\cdot \log \frac1{\gamma_0 \eta}\right)\\
\leq &\tilde{\gamma}\gamma\left(\eta+
5\eta\log \frac1{\gamma_0 \eta}\right).
\end{align*}
According to Claim~\ref{tlogt-monotonicity}, we may choose
$\eta_\varepsilon>0$ and $t_\varepsilon=\tilde{t}_{\eta_\varepsilon}$ for \eqref{variation-phi-thL-Ithetat-eta}  such that $\eta_\varepsilon\leq \min\{ \frac{M}2,\frac1{\gamma_0R_0},\varepsilon\}$ and $\tilde{\gamma}\gamma\left(\eta_\varepsilon+
5\eta_\varepsilon\log \frac1{\gamma_0 \eta_\varepsilon}\right)\leq \varepsilon$, completing the proof of \eqref{variation-phi-thL-Ithetat-eta}.

To prove \eqref{variation-phi-thL-Ithetat-limit}, we plan to use Lebesgue's Dominant Convergence Theorem where \eqref{variation-phi-thL-gamma0-M-s} yields that for $r_{\eta_\varepsilon}=\gamma_0\eta_\varepsilon$, we have
\begin{equation}
\label{variation-phi-thL-retaeps}
 r_{\eta_\varepsilon}\,B^n\subset F_s\mbox{ \ if }s\in\left(0,M-\eta_\varepsilon\right].
\end{equation}
Here $t_\varepsilon=\tilde{t}_{\eta_\varepsilon}\leq \frac{\gamma_0\eta_\varepsilon}{4m\sigma}$ yields that if $0<|t|<t_\varepsilon$, then
\begin{equation*}
\label{variation-phi-thL-t-less-teps}
|t|\sigma\leq  \frac{r_{\eta_\varepsilon}}{4m}  \mbox{ \ and \ }\frac{r_{\eta_\varepsilon}}2\,B^n\subset \{f_t\geq s\}=F_s\oplus t\cdot L\mbox{ \ if }s\in\left(0,M-\eta_\varepsilon\right],
\end{equation*}
and hence $\varrho_{\{f_t\geq s\}}(u)>0$ for $u\in S^{n-1}$ and  $s\in(0,M-\eta_\varepsilon]$. 
If $0<|t|<t_\varepsilon$, $u\in S^{n-1}$ and  $s\in(0,M-\eta_\varepsilon]$, then we deduce from \eqref{variation-phi-thL-retaeps}, $\{f_t\geq s\}=F_s\oplus t\cdot L$ and $L\subset \sigma B^n$ that
$$
\left(1-\frac{\sigma|t|}{r_{\eta_\varepsilon}}\right)^m\varrho_{F_s}(u)^m\leq \varrho_{\{f_t\geq s\}}(u)^m\leq \left(1+\frac{\sigma|t|}{r_{\eta_\varepsilon}}\right)^m\varrho_{F_s}(u)^m,
$$
thus \eqref{1+tau-m} and \eqref{variation-phi-thL-fxi-tautxi2} imply that  (cf. \eqref{FsL-fts-upsilon-pos})
\begin{equation}
\label{variation-phi-thL-IMt-pointwise-upp}
0<b_{t,\xi}^{q-1}\cdot \Upsilon(u,t,s)\leq \frac{2m\sigma A}{r_{\eta_\varepsilon}}\cdot \varrho_{F_s}(u)^m.
\end{equation}
On the other hand, first \eqref{rho-tom-xi-inner-int}, then \eqref{variation-phi-thL-retaeps}, after  that $m\geq 1$, and finally the layer cake formula yield that
\begin{align*}
I^*_{\eta_\varepsilon}= &
\int_0^{M-\eta_\varepsilon}\int_{{\rm G}(n,m)}\int_{S^{n-1}\cap\xi}\varrho_{F_s}(u)^m\,d\HH^{m-1}(u)\,d\xi\,ds\\
=& \frac{m^2\omega_m }{n\omega_n}\left[\int_0^{M-\eta_\varepsilon}\int_{r_{\eta_\varepsilon}B^n}\|x\|^{m-n}\,dx\,ds+\int_0^{M-\eta_\varepsilon}\int_{F_s\backslash r_{\eta_\varepsilon}B^n}\|x\|^{m-n}\,dx\,ds\right]\\
\leq & m\omega_mM\cdot r_{\eta_\varepsilon}^m+\frac{m^2\omega_m }{n\omega_n}\cdot r_{\eta_\varepsilon}^{m-n}\int_0^M|F_s|\,ds<\infty.
\end{align*}
Combining this estimate with \eqref{variation-phi-thL-IMt-pointwise-upp} shows that we can apply Lebesgue's Dominant Convergence Theorem to \eqref{variation-phi-thL-Ithetat-limit}. To get the pointwise limit, for each $s\in(0,M-\eta_\varepsilon]$ and for $\HH^{n-1}$ a.e. $u\in S^{n-1}$, \eqref{HLYZ-radial-der} (cf. \eqref{log-sum-level}) and
\eqref{variation-phi-thL-fxi-tautxi-limit}
 imply that
$$
\lim_{t\to 0}b_{t,\xi}^{q-1}\cdot \Upsilon(u,t,s)=m\left\|f|_\xi\right\|_1^{q-1}\cdot
\frac{h_L(\alpha_{F_s}(u))\cdot \varrho_{F_s}(u)^m}{h_{F_s}(\alpha_{F_s}(u))},
$$
and hence Lebesgue's Dominant Convergence Theorem yields that
$\lim_{t\to 0}I(M-\eta_\varepsilon,t)=J(M-\eta_\varepsilon)$ as required by \eqref{variation-phi-thL-Ithetat-limit}.

Having proved \eqref{variation-phi-thL-Ithetat-eta} and \eqref{variation-phi-thL-Ithetat-limit}, we are ready to complete Step 1. 
First of all, for any $\varepsilon \in (0,1)$, \eqref{variation-phi-thL-Ithetat-eta} and \eqref{variation-phi-thL-Ithetat-limit} imply the existence of a $t^*_\varepsilon\in(0,t_\varepsilon)$ such that
if $0<|t|\leq t^*_\varepsilon$, then $|I(M,t)-I(M-\eta_\varepsilon,t)|\leq \varepsilon$ and  
$\left|I(M-\eta_\varepsilon,t)-J(M-\eta_\varepsilon)\right|\leq \varepsilon$; therefore, as $\eta_\varepsilon\leq \varepsilon$ by definition,
\begin{align*}
\left|I(M,t)-J(M)\right|\leq&|I(M,t)-I(M-\eta_\varepsilon,t)|+\left|I(M-\eta_\varepsilon,t)-J(M-\eta_\varepsilon)\right|+\\
&+|J(M)-J(M-\eta_\varepsilon)|
\leq 2\varepsilon+|J(M)-J(M-\varepsilon)|.
\end{align*}
Here $J(M)<\infty$ (cf. \eqref{JM-finite}) yields $\lim_{\varepsilon\to 0^+}\left[2\varepsilon+|J(M)-J(M-\varepsilon)|\right]=0$, thus proving 
$\lim_{t\to 0}I(M,t)=J(M)$, as it is required by Step 1 (cf. \eqref{variation-phi-thL-tildePsimq-Ithetat}).\\

\noindent{\bf Step 2.} $J(M)=q\int_{\Rn}h_L\,d\widetilde{A}_{m,q}^{e}(f,\cdot)+q\int_{S^{n-1}}h_L\,d\widetilde{A}_{m,q}^{s}(f,\cdot)$.

In order to apply \eqref{HLYZ-cone-volume0}, we can extend the definition of the dual Radon transform in a way such that if $\lambda\neq 0$ and $u\in S^{n-1}$, then
$$
\widetilde{\mathcal{R}}_m^* \left(({\rm R}_mf)^{q-1}\right)(\lambda u)=\widetilde{\mathcal{R}}_m^* \left(({\rm R}_mf)^{q-1}\right)(u).
$$

We deduce first from applying the duality formula \eqref{Rm-dualRm-sphere} for any given $s\in(0,M)$, then from \eqref{HLYZ-cone-volume0} that
\begin{align}
\nonumber
J(M)=&q\int_0^M\int_{{\rm G}(n,m)}{\rm R}_mf(\xi)^{q-1}\widetilde{\mathcal{R}}_m \left(\varrho_{F_s}^m
\cdot \frac{h_L\circ \alpha_{F_s}}{h_{F_s}\circ \alpha_{F_s}}\right)\,d\xi\,ds\\
\nonumber
=&q\int_0^M\int_{S^{n-1}}\widetilde{\mathcal{R}}_m^* \left(({\rm R}_mf)^{q-1}\right)\cdot\varrho_{F_s}^m
\cdot \frac{h_L\circ \alpha_{F_s}}{h_{F_s}\circ \alpha_{F_s}}\,d\HH^{n-1}\,ds\\
\label{JM-on-sphere}
=&q\int_0^M\int_{\partial F_s} h_L(u_{F_s}(x))\cdot \|x\|^{m-n}
\widetilde{\mathcal{R}}_m^* \left(({\rm R}_mf)^{q-1}\right)\left(\frac{x}{\|x\|}\right)
\,dx\,ds.
\end{align}
Here if $s\in (0,M)$ and $x\in\partial F_s$, then
$$
h_L(u_{F_s}(x))=\left\{
\begin{array}{rcl}
\frac{h_L(\nabla\varphi(x))}{\|\nabla\varphi(x)\|}&\mbox{ if }&x\in{\rm int}\,D_f\mbox{ and }\varphi\mbox{ is differentiable at }x,\\[1ex]
h_L(u_{D_f}(x))&\mbox{ if }&x\in(\partial'D_f)\cap(\partial'F_s).
\end{array}\right.
$$
It follows from applying the coarea formula \eqref{coarea-BV-eq} and Claim~\ref{exterior-normal} to \eqref{JM-on-sphere} that
\begin{align*}
J(M)=&q\int_{\R^n} h_L(\nabla\varphi(x)) \|x\|^{m-n}f(x) \left(\widetilde{\mathcal{R}}^*_m({\rm R}_mf)^{q-1}\right)\left(\frac{x}{\|x\|}\right)\,dx\\
&+q\int_{\partial D_f} h_L(u_{D_f}(x))\|x\|^{m-n}f(x) \left(\widetilde{\mathcal{R}}^*_m({\rm R}_mf)^{q-1}\right)\left(\frac{x}{\|x\|}\right)\,dx,
\end{align*}
completing the proof of Step~2, and in turn Proposition~\ref{variation-phi+thL} if $q\neq 0$ by \eqref{tildeAdef-Rn} and \eqref{tildeAdef-s}.

If $q=0$ in Proposition~\ref{variation-phi+thL}, then the only change in the argument is that \eqref{variation-phi-thL-fxi-tautxi1} is replaced by
\begin{equation*}
\label{variation-phi-thL-fxi-tautxi0}
\log \left\|f_t|_\xi\right\|_1-\log \left\|f|_\xi\right\|_1=b_{t,\xi}^{-1}\cdot \left(\left\|f_t|_\xi\right\|_1-\left\|f|_\xi\right\|_1\right),
\end{equation*}
completing the proof of Proposition~\ref{variation-phi+thL}.
\end{proof}

\begin{coro}
\label{variation-phi+tchL}
Let  $q\in\R$, $c>0$ and $m\in\{1,\ldots,n-1\}$. Let $f=e^{-\varphi}$ be an upper semicontinuous log-concave function  on $\Rn$ such that $0<\int_{\Rn} f<\infty$, $o\in{\rm int}\,D_f$ and $f(o)=\sup f>0$. Let $L\subset \R^n$ be a compact convex set with $o\in L$. Then $\psi^*=h_L+\log c$ holds for the log-concave function $c\cdot\mathbf{1}_L=e^{-\psi}$, and  $f_t=e^{-\varphi_t}$ for $\varphi_t=(\varphi^*+t\psi^*)^*$  and $t\in\R$ satisfies that
\begin{align}
\label{variation-phi+tchL-eq}
\lim_{t\to 0}\frac{\widetilde{\Psi}_{m,q}(f_t)-\widetilde{\Psi}_{m,q}(f)}t
=&q\int_{\Rn}\psi^*\,d\widetilde{A}_{m,q}^{e}(f,\cdot)+q\int_{S^{n-1}}h_L\,d\widetilde{A}_{m,q}^{s}(f,\cdot),\;q\neq 0; \\
\label{variation-phi+tchL-eq0}
\lim_{t\to 0}\frac{\widetilde{\Psi}_{m,0}(f_t)-\widetilde{\Psi}_{m,0}(f)}t
=&\int_{\Rn}\psi^*\,d\widetilde{A}_{m,0}^{e}(f,\cdot)+\int_{S^{n-1}}h_L\,d\widetilde{A}_{m,0}^{s}(f,\cdot),
\end{align}
where both integrals occurring in \eqref{variation-phi+tchL-eq} and in \eqref{variation-phi+tchL-eq0} are finite. 
\end{coro}
\noindent{\bf Remark.} 
Here $h_L$ is the recession function $\bar{\psi}^*$ of $\psi^*$ (cf. \eqref{barpsi-dompsi-bounded}).
\begin{proof}
We have $\psi(x)=-\log c$ if $x\in L$ and $\psi(x)=\infty$ if $x\not\in L$, and hence $\psi^*=h_L+\log c$ and $f_t=c^t\tilde{f}_t$ for $\tilde{f}_t=e^{-(\varphi^*+th_L)^*}$ (cf. Example~\ref{example-c1L}).

First, let $q \neq 0$, thus $\widetilde{\Psi}_{m,q}(f_t)=c^{qt}\widetilde{\Psi}_{m,q}(\tilde{f}_t)$ for $t\in\R$. We deduce from Proposition~\ref{variation-phi+thL} and $\widetilde{A}_{m,q}^{e}(f,\R^n)= \widetilde{\Psi}_{m,q}(f)$ (cf.  \eqref{tildeA-finiteness-RntildePsi} in Proposition~\ref{tildeA-finiteness})  that
\begin{align*}
\lim_{t\to 0}\frac{\widetilde{\Psi}_{m,q}(f_t)-\widetilde{\Psi}_{m,q}(f)}t
=& \lim_{t\to 0}\frac{c^{qt}-1}t \cdot \widetilde{\Psi}_{m,q}(\tilde{f}_t)
+\lim_{t\to 0}\frac{\widetilde{\Psi}_{m,q}(\tilde{f}_t)-\widetilde{\Psi}_{m,q}(f)}t\\
=&q\int_{\Rn}(h_L+\log c)\,d\widetilde{A}_{m,q}^{e}(f,\cdot)+q\int_{S^{n-1}}h_L\,d\widetilde{A}_{m,q}^{s}(f,\cdot)\\
=&q\int_{\Rn}\psi^*\,d\widetilde{A}_{m,q}^{e}(f,\cdot)+q\int_{S^{n-1}}h_L\,d\widetilde{A}_{m,q}^{s}(f,\cdot).
\end{align*}

If $q=0$, then $\widetilde{\Psi}_{m,0}(f_t)=t\log c+\widetilde{\Psi}_{m,0}(\tilde{f}_t)$ for $t\in\R$. We deduce from Proposition~\ref{variation-phi+thL} and $\widetilde{A}_{m,0}^{e}(f,\R^n)= 1$ (cf.  \eqref{tildeA-finiteness-RntildePsi0} in Proposition~\ref{tildeA-finiteness})  that
\begin{align*}
\lim_{t\to 0}\frac{\widetilde{\Psi}_{m,0}(f_t)-\widetilde{\Psi}_{m,0}(f)}t
=& \log c
+\lim_{t\to 0}\frac{\widetilde{\Psi}_{m,0}(\tilde{f}_t)-\widetilde{\Psi}_{m,0}(f)}t\\
=&\int_{\Rn}(h_L+\log c)\,d\widetilde{A}_{m,0}^{e}(f,\cdot)+\int_{S^{n-1}}h_L\,d\widetilde{A}_{m,0}^{s}(f,\cdot),
\end{align*}
completing the proof of Corollary~\ref{variation-phi+tchL}.
\end{proof}
% \begin{proof}[Proof of Proposition~\ref{Amq-pair-measures}]
% \end{proof}
In Proposition~\ref{variation-phi+tzeta}, we use the following consequence of Proposition~2.1 in Rotem \cite{R22}.

\begin{lemma}[Rotem]
\label{limt-variation-diffphi}
Let $\varphi,\zeta:\R^n \rightarrow (-\infty, \infty]$ be lower semicontinuous functions bounded from below such that $\varphi$ is convex and $\varphi^*(o),\zeta(o)<\infty$. If $\varphi$ is differentiable at an $x\in\R^n$, then
\begin{equation}
\label{limt-variation-diffphi-eq}
\lim_{t\to 0^+}\frac{(\varphi^*+t\zeta)^*(x)-\varphi(x)}t=-\zeta(\nabla \varphi(x)).
\end{equation}
If, in addition, if $\zeta$ is continuous and bounded, then
\begin{equation}
\label{limt-variation-diffphi0-eq}
\lim_{t\to 0}\frac{(\varphi^*+t\zeta)^*(x)-\varphi(x)}t=-\zeta(\nabla \varphi(x)).
\end{equation}
\end{lemma}
\noindent{\bf Remark.} Proposition~2.1 in Rotem \cite{R22} actually only states \eqref{limt-variation-diffphi-eq}, but as the paper later remarks, if $\zeta$ is continuous and bounded, then applying \eqref{limt-variation-diffphi-eq} to $-\zeta$, as well, yields \eqref{limt-variation-diffphi0-eq}.  

\begin{prop}
\label{variation-phi+tzeta}
Let $q\in\R$ and $m\in\{1,\ldots,n-1\}$. Let $f=e^{-\varphi}$ be an upper semicontinuous log-concave function  on $\Rn$ such that $0<\int_{\Rn} f<\infty$, $o\in{\rm int}\,D_f$ and $f(o)=\sup f>0$. Let $\zeta\in C_c(\R^n)$. Then $f_t=e^{-\varphi_t}$ for $\varphi_t=(\varphi^*+t\zeta)^*$ and $t\in\R$ satisfies that
\begin{align}
\label{variation-phi+tzeta-eq}
\lim_{t\to 0}\frac{\widetilde{\Psi}_{m,q}(f_t)-\widetilde{\Psi}_{m,q}(f)}t
=&q\int_{\Rn}\zeta\,d\widetilde{A}_{m,q}^{e}(f,\cdot),\mbox{ \ \ }q\neq 0;\\ 
\label{variation-phi+tzeta0-eq}
\lim_{t\to 0}\frac{\widetilde{\Psi}_{m,0}(f_t)-\widetilde{\Psi}_{m,0}(f)}t
=&\int_{\Rn}\zeta\,d\widetilde{A}_{m,0}^{e}(f,\cdot),
\end{align}
where the integrals occurring \eqref{variation-phi+tzeta-eq} and \eqref{variation-phi+tzeta0-eq} are finite.
\end{prop}
\begin{proof} 
During the argument, we breviate $d\nu_{n,m}(\xi)$ to $d\xi$ when integrating on ${\rm G}(n,m)$. There exists some $N>0$ such that
\begin{equation}
\label{variation-phi-tzeta-zetaN}
|\zeta(x)|\leq N \mbox{ \  for }x\in\R^n,\mbox{ \ and hence }e^{-N|t|}f(x)\leq f_t(x)\leq e^{N|t|}f(x)
\end{equation}
holds (cf. \eqref{Legendre-shift} and \eqref{Legendre-monotone}) for any $x\in\R^n$ and $t\in\R$. We deduce the finiteness of
the integrals occurring in \eqref{variation-phi+tzeta-eq} and in \eqref{variation-phi+tzeta0-eq} from Proposition~\ref{tildeA-finiteness}.

We choose a $t_0\in(0,\frac1{2N}]$ depending on $f$ and  $\zeta$ such that Corollary~\ref{Taylorqb} holds for $t\in(-t_0,t_0)$. According to Corollary~\ref{Taylorqb}, there exist $A>0$ depending on $f$ and $\zeta$ such that if $t\in(-t_0,t_0)$ and $\xi\in{\rm G}(n,m)$, then
\begin{equation}
\label{variation-phi-tzeta-ftxi-q-1}
\left\|f_t|_\xi\right\|_1^{q-1}\leq A.
\end{equation}

First, let $q\neq 0$. Now Corollary~\ref{Taylorqb} yields the existence of $A>0$ depending on $f$ and $\zeta$ such that if  $t\in(-t_0,t_0)$  and $\xi\in{\rm G}(n,m)$, then there exists a $b_{t,\xi}\in\R$ satisfying
\begin{align}
\label{variation-phi-tzeta-fxi-tautxi1}
\left\|f_t|_\xi\right\|_1^q-\left\|f|_\xi\right\|_1^q=&q\cdot b_{t,\xi}^{q-1}\cdot \left(\left\|f_t|_\xi\right\|_1-\left\|f|_\xi\right\|_1\right),\\
\label{variation-phi-tzeta-fxi-btxi-continuous}
b_{t,\xi}\mbox{ is a continuous function}&\mbox{ of $\xi\in{\rm G}(n,m)$ for a fixed $t\in(-t_0,t_0)$,}\\
\label{variation-phi-tzeta-fxi-tautxi2}
b_{t,\xi}^{q-1}\leq &A,\\
\label{variation-phi-thL-fxi-tautxi3}
\lim_{t\to 0}b_{t,\xi}=&({\rm R}_mf).
\end{align}
For fixed $t$ with $0<|t|<t_0$, we deduce from
 \eqref{variation-phi-tzeta-fxi-tautxi1} and  \eqref{fxi-representations-Ball-eq} in Lemma~\ref{fxi-representations} that 
\begin{align}
\nonumber
\frac{\widetilde{\Psi}_{m,q}(f_t)-\widetilde{\Psi}_{m,q}(f)}t=&\int_{G(n,m)}\frac{\|f_t|_{\xi}\|_1^q-\|f|_{\xi}\|_1^q}t\,d\xi\\
\label{variation-phi-tzta-tildePsimq-diff}
=&q\int_{{\rm G}(n,m)}\int_{\xi} b_{t,\xi}^{q-1}  \cdot \frac{f_t(x)-f(x)}t\,dx\,d\xi.
\end{align}
To calculate the limit of the double integral in \eqref{variation-phi-tzta-tildePsimq-diff}, we apply Lebesgue's Dominant Convergence theorem. For the integrable pointwise upper bound, we deduce from \eqref{variation-phi-tzeta-zetaN} and \eqref{variation-phi-tzeta-fxi-tautxi2} that if $t\in(-t_0,t_0)$ (and hence $N|t|\leq 1$), $\xi\in G(n,m)$ and $x\in\xi$, then
$$
b_{t,\xi}^{q-1} \cdot \left|\frac{f_t(x)-f(x)}t\right|\leq 2ANf(x),
$$
where  
$$
\int_{{\rm G}(n,m)}\int_{\xi}2ANf(x)\,dx\,d\xi=
2AN\widetilde{\Psi}_{m,1}(f)<\infty.
$$
On the other hand, $\varphi^*$ is finite in a neighbourhood of the origin as $\varphi$ is a coercive, and hence
\eqref{variation-phi-thL-fxi-tautxi3} and Lemma~\ref{limt-variation-diffphi} imply that if
$\xi\in G(n,m)$ and $\varphi$ is differentiable at an $x\in\xi$, then
$$
\lim_{t\to 0}b_{t,\xi}^{q-1}  \cdot \frac{f_t(x)-f(x)}t=\left\|f|_\xi\right\|_1^{q-1}  \cdot \zeta(\nabla\varphi(x))\cdot f(x).
$$
Since $\varphi$ is $\HH^n$ a.e. differentiable at the points of a Borel set on $\R^n$, the duality formula \eqref{Rm-dualRm-sphere} yields that $\nu_{n,m}$ a.e. $\xi\in {\rm G}(n,m)$ satisfies the property that $\varphi$ is differentiable at $\HH^m$ a.e. point of $\xi$. In turn, we conclude via Lebesgue's Dominant Convergence theorem and Lemma~\ref{tildeA-e-calculation} that
\begin{align*}
\lim_{t\to 0}\frac{\widetilde{\Psi}_{m,q}(f_t)-\widetilde{\Psi}_{m,q}(f)}t=&q\int_{{\rm G}(n,m)}\int_{\xi}
\left\|f|_\xi\right\|_1^{q-1}  \cdot \zeta(\nabla\varphi(x))\cdot f(x)\,dx\,d\xi\\
=&q\int_{\Rn}\zeta\,d\widetilde{A}_{m,q}^{e}(f,\cdot).
\end{align*}

If $q=0$ in Proposition~\ref{variation-phi+tzeta}, then the only change in the argument is that \eqref{variation-phi-tzeta-fxi-tautxi1} is replaced by
\begin{equation*}
\label{variation-phi-thL-fxi-tautxi0}
\log \left\|f_t|_\xi\right\|_1-\log \left\|f|_\xi\right\|_1=b_{t,\xi}^{-1}\cdot \left(\left\|f_t|_\xi\right\|_1-\left\|f|_\xi\right\|_1\right),
\end{equation*}
completing the proof of Proposition~\ref{variation-phi+tzeta}.
\end{proof}

\section{The proof of Theorem~\ref{Amq-pair-measures} }
\label{secTheorem1.1}

In order to simplify some formulas, for a log-concave function $f$ on $\R^n$ with $0<\int_{\Rn} f < \infty$, we set
\begin{equation}
\label{eq:thm11-normalized-functional}
 \widetilde{\Psi}^*_{m,q}(f)=
 \begin{cases}
  \dfrac{1}{q}\,\widetilde{\Psi}_{m,q}(f),&q\ne0,\\
  \widetilde\Psi_{m,0}(f),&q=0,
 \end{cases}
\end{equation}
and  for $q\in\R$ and $s>0$ we define
\[
 \Phi_q(s)=
 \begin{cases}
  s^q/q,&q\ne0,\\
  \log s,&q=0,
 \end{cases}
\]
and hence $\Phi_q$ is increasing for all $q\in\R$. It follows that for $q\in\R$ and
upper semicontinuous log-concave function $f=e^{-\varphi}$ on $\R^n$ with $0<\int_{\Rn} f < \infty$, we have
\begin{equation}
\label{Psi*Phidef}
 \widetilde{\Psi}^*_{m,q}(f)
 =\int_{{\rm G}(n,m)}\Phi_q\bigl(\|f|_\xi\|_1\bigr)\,d\nu_{m}(\xi),
\end{equation}
and  Corollary~\ref{variation-phi+tchL} yields that if $c>0$, $L\subset \R^n$ is a compact convex set with $o\in L$, and $c\mathbf{1}_L=e^{-\psi}$, then $f_t=f\oplus t\cdot (c\mathbf{1}_L)=e^{-(\varphi^*+t\psi^*)^*}$ for $t\ge  0$ satisfies that  
\begin{equation}
\label{variation-phi+tchL-again}
\lim_{t\to 0^+}\frac{\widetilde{\Psi}^*_{m,q}(f_t)-\widetilde{\Psi}^*_{m,q}(f)}t
=\int_{\Rn}\psi^*\,d\widetilde{A}_{m,q}^{e}(f,\cdot)+\int_{S^{n-1}}h_{L}\,d\widetilde{A}_{m,q}^{s}(f,\cdot).
\end{equation}

\begin{proof}[Proof of Theorem~\ref{Amq-pair-measures}]
We recall that $q\in\R$, $m\in\{1,\ldots,n-1\}$, $f=e^{-\varphi}$ and $g=e^{-\psi}$ are
upper semicontinuous log-concave functions  on $\R^n$ such that  $0<\int_{\Rn} f < \infty$, $o\in{\rm int}\,D_f$ and $f(o)=\max f$, while $g(o)>0$ and $g$ has compact support. Let $L={\rm cl}D_g$, that is a compact convex set with $o\in L$.
 In addition, for any integer $k>g(o)^{-1}$, let $L_k=\{x\in L:g(x)\geq \frac1k\}$, which is a compact convex set with $o\in L_k$. Since the set of points $x\in{\rm int}\,D_f$ where $\varphi$ is not differentiable is a Borel set of zero $\HH^n$ measure, we deduce from the duality formula \eqref{Rm-dualRm-sphere} that for $\nu_{m}$ a.e. $\xi\in{\rm G}(n,m)$, 
\begin{equation}
\label{phi-diff-on-xi} 
\varphi\mbox{ is }\HH^m\mbox{ a.e. differentiable on }\xi.
\end{equation}
For $f_t = f\oplus(t\cdot g)=e^{-(\varphi^*+t\psi^*)^*}$, the core claims are that 
\begin{align}
\label{logconv-g-core-upper}
\limsup_{t\to 0^+}\frac{\widetilde{\Psi}^*_{m,q}(f_t)-\widetilde{\Psi}^*_{m,q}(f)}t
\leq &\int_{\Rn}\psi^*\,d\widetilde{A}_{m,q}^{e}(f,\cdot)+\int_{S^{n-1}}h_{L}\,d\widetilde{A}_{m,q}^{s}(f,\cdot),\\
\label{logconv-g-core-lower}
\liminf_{t\to 0^+}\frac{\widetilde{\Psi}^*_{m,q}(f_t)-\widetilde{\Psi}^*_{m,q}(f)}t
\geq &\int_{\Rn}\psi^*\,d\widetilde{A}_{m,q}^{e}(f,\cdot)+\int_{S^{n-1}}h_{L_k}\,d\widetilde{A}_{m,q}^{s}(f,\cdot).
\end{align}
Since $g$ is compactly supported, there exists $C>0$ such that $g\leq C 1_{L}:=e^{-\tilde{\psi}}$, and for $k>g(o)^{-1}$, let  $\frac1k\, 1_{L_k}=e^{-\hat{\psi}_k}\leq g$. We consider 
\begin{align}
\label{tildef-f}
\tilde{f}_t = &f\oplus(t\cdot C 1_{L})=e^{-(\varphi^*+t\tilde{\psi}^*)^*}\geq f_t,\\
\label{hatf-f}
\hat{f}_{k,t} = &f\oplus(t\cdot \mbox{$\frac1k$}\, 1_{L_k})=e^{-(\varphi^*+t\hat{\psi}_k^*)^*}\leq f_t.
\end{align}
We choose $\sigma>0$ such that $L\subset\sigma B^n$, and hence $L_k\subset\sigma B^n$, as well, and we choose $t_0>0$ coming from Corollary~\ref{Taylorqb} that depends on $q$, $\sigma$, $f$, and works for both $c=C$ or $c=\frac1k$. 
Applying the Taylor formula to $\Phi_q$, we deduce that for any $t\in(-t_0,t_0)$ and $\xi\in{\rm G}(n,m)$, there exist $\tilde{b}_{t,\xi}$ and $\hat{b}_{k,t,\xi}$ with $\left\|f_t|_\xi\right\|_1\leq \tilde{b}_{t,\xi}\leq \left\|\tilde{f}_t|_\xi\right\|_1$ and $\left\|\hat{f}_{k,t}|_\xi\right\|_1\leq \hat{b}_{k,t,\xi}\leq \left\|f_t|_\xi\right\|_1$ such that
\begin{align}
\label{tildef-f-b-def}
\Phi_q\left(\left\|\tilde{f}_t|_\xi\right\|_1\right)-\Phi_q\left(\left\|f_t|_\xi\right\|_1\right)=& \tilde{b}_{t,\xi}^{q-1}\cdot \left(\left\|\tilde{f}_t|_\xi\right\|_1-\left\|\tilde{f}|_\xi\right\|_1\right),\\
\label{hatf-f-b-def}
\Phi_q\left(\left\|f_t|_\xi\right\|_1\right)-\Phi_q\left(\left\|\hat{f}_{k,t}|_\xi\right\|_1\right)=& \hat{b}_{k,t,\xi}^{q-1}\cdot \left(\left\|f_t|_\xi\right\|_1-\left\|\hat{f}_{k,t}|_\xi\right\|_1\right).
\end{align}
We deduce that $\tilde{b}_{t,\xi}$ and $\hat{b}_{k,t,\xi}$ are continuous functions of $\xi\in{\rm G}(n,m)$ for a fixed $t\in(-t_0,t_0)$, and if  $\xi\in{\rm G}(n,m)$, then 
Corollary~\ref{Taylorqb} and $\left\|\hat{f}_{k,t}|_\xi\right\|_1\leq \hat{b}_{k,t,\xi} \leq \tilde{b}_{t,\xi}\leq \left\|\tilde{f}_t|_\xi\right\|_1$ yield that
\begin{equation}
\label{Th11-b-limit}
\lim_{t\to 0^+}\tilde{b}_{t,\xi}=\lim_{t\to 0^+}\hat{b}_{k,t,\xi}=\left\|f|_\xi\right\|_1. 
\end{equation}

In order to apply Lemma~\ref{limt-variation-diffphi}, we observe that $\psi^*$, $\tilde{\psi}^*$ and $\hat{\psi}^*_{k}$ are lower semicontinuous,  $\tilde{\psi}^*=h_L+\log C$ and $\hat{\psi}^*_{k}=h_{L_k}+\log \frac1k$ are bounded from below because $o\in L_k\subset L$, and $\psi^*$ is also bounded from below because $\psi(o)<\infty$. We deduce from Lemma~\ref{limt-variation-diffphi} that if $\varphi$ is differentiable at an $x\in\Rn$, then
\begin{align}
\label{Th11-Rotem-limit-psi}
\lim_{t\to 0^+}\frac{f_t(x)-f(x)}t=&\psi^*(\nabla \varphi(x))\cdot f(x),\\
\label{Th11-Rotem-limit-tildepsi}
\lim_{t\to 0^+}\frac{\tilde{f}_t(x)-f(x)}t=&\tilde{\psi}^*(\nabla \varphi(x))\cdot f(x),\\
\label{Th11-Rotem-limit-hatpsi}
\lim_{t\to 0^+}\frac{\hat{f}_{k,t}(x)-f(x)}t=&\hat{\psi}_{k}^*(\nabla \varphi(x))\cdot f(x).
\end{align}
As $o\in L_k\subset L\subset \sigma B^n$, writing $\beta_0=\max\{|\log C|,|\log k|\}$, the condition $\hat{\psi}^*_{k}\leq \psi\leq \tilde{\psi}$ following from $\frac1k\mathbf{1}_{L_k}\leq g\leq C\mathbf{1}_L$ yields that
\begin{equation}
\label{Th11-psi-etc-upper}
|\hat{\psi}_{k}^*(x)|,\;|\psi(x)|,\;|\tilde{\psi}(x)|\leq \sigma\|x\|+\beta_0\mbox{ \ for any }x\in\R^n.
\end{equation}

To prove the core claim \eqref{logconv-g-core-upper}, 
we note that $\tilde{f}_t\geq f_t$ and
 $\widetilde{\Psi}^*_{m,q}(\tilde{f}_t)-\widetilde{\Psi}^*_{m,q}(f_t)\geq 0$ as $\Phi_q$ is monotone increasing (cf. \eqref{Psi*Phidef}), and hence we can apply Fatou's lemma to $\frac{\tilde{f}_t-f_t}t$.
We also note that if $\varphi$ is differentiable at an $x\in\Rn$, then \eqref{Th11-Rotem-limit-psi} and \eqref{Th11-Rotem-limit-tildepsi} yield that 
\begin{equation}
\label{T11-limit-tildef-f-difference}
\lim_{t\rightarrow 0^{+}}\frac{\tilde{f}_t(x)-f_t(x)}{t}=
\left(\tilde{\psi}^*(\nabla \varphi(x))-\psi^*(\nabla \varphi(x))\right)\cdot f(x).
\end{equation}
We deduce from \eqref{Psi*Phidef} and \eqref{tildef-f-b-def}, then from Fatou's lemma, after that from \eqref{phi-diff-on-xi}, 
\eqref{Th11-b-limit} and \eqref{T11-limit-tildef-f-difference}, and finally from Lemma~\ref{tildeA-e-calculation} (cf. \eqref{Th11-psi-etc-upper}) that 
\begin{align*}
\liminf_{t\rightarrow 0^{+}}\frac{\widetilde{\Psi}^*_{m,q}(\tilde{f}_t)-\widetilde{\Psi}^*_{m,q}(f_t)}t=
&\liminf_{t\rightarrow 0^{+}}\int_{{\rm G}(n,m)}\int_{\xi}\tilde{b}_{t,\xi}^{q-1}\cdot\frac{\tilde{f}_t(x)-f_t(x)}{t}dxd\xi \\
\geq& \int_{{\rm G}(n,m)}\int_{\xi} \liminf_{t\rightarrow 0^{+}}\tilde{b}_{t,\xi}^{q-1}\cdot\frac{\tilde{f}_t(x)-f_t(x)}{t}dxd\xi\\
=&\int_{{\rm G}(n,m)}\int_{\xi}
\left\|f|_\xi\right\|_1^{q-1}\left(\tilde{\psi}^*(\nabla\varphi)-\psi^*(\nabla\varphi)\right) f\,d\HH^m\,d\xi\\
=&\int_{\Rn}\left(\tilde{\psi}^*-\psi^*\right)\,d\widetilde{A}_{m,q}^{e}(f,\cdot).
\end{align*}
Therefore, \eqref{variation-phi+tchL-again} implies that
\begin{align*}
\limsup_{t\rightarrow 0^{+}}\frac{\widetilde{\Psi}^*_{m,q}(f_t)-\widetilde{\Psi}^*_{m,q}(f)}t=&\lim_{t\rightarrow 0^{+}}\frac{\widetilde{\Psi}^*_{m,q}(\tilde{f}_t)-\widetilde{\Psi}^*_{m,q}(f)}t\\
&-\liminf_{t\rightarrow 0^{+}}\frac{\widetilde{\Psi}^*_{m,q}(\tilde{f}_t)-\widetilde{\Psi}^*_{m,q}(f_t)}t\\
\leq
& \int_{\Rn}\tilde{\psi}^*\,d\widetilde{A}_{m,q}^{e}(f,\cdot)+\int_{S^{n-1}}h_L\,d\widetilde{A}_{m,q}^{s}(f,\cdot)\\
&-\int_{\Rn}\left(\tilde{\psi}^*-\psi^*\right)\,d\widetilde{A}_{m,q}^{e}(f,\cdot)\\
=&\int_{\Rn}\psi^*\,d\widetilde{A}_{m,q}^{e}(f,\cdot)+\int_{S^{n-1}}h_L\,d\widetilde{A}_{m,q}^{s}(f,\cdot),
\end{align*}
proving \eqref{logconv-g-core-upper}.

To prove the other core claim \eqref{logconv-g-core-lower}, we apply a similar argument.
We deduce from \eqref{Psi*Phidef} and \eqref{hatf-f-b-def}, then from Fatou's lemma, after that from \eqref{phi-diff-on-xi}, \eqref{Th11-b-limit}, \eqref{Th11-Rotem-limit-psi} and \eqref{Th11-Rotem-limit-hatpsi}, and finally from Lemma~\ref{tildeA-e-calculation} (cf. \eqref{Th11-psi-etc-upper}) that 
\begin{align*}
\liminf_{t\rightarrow 0^{+}}\frac{\widetilde{\Psi}^*_{m,q}(f_t)-\widetilde{\Psi}^*_{m,q}(\hat{f}_{k,t})}t
=&\liminf_{t\rightarrow 0^{+}}\int_{{\rm G}(n,m)}\int_{\xi}\hat{b}_{k,t,\xi}^{q-1}\cdot\frac{f_t(x)-\hat{f}_{k,t}(x)}{t}dxd\xi \\
\geq& \int_{{\rm G}(n,m)}\int_{\xi} \liminf_{t\rightarrow 0^{+}}\hat{b}_{k,t,\xi}^{q-1}\cdot\frac{f_t(x)-\hat{f}_{k,t}(x)}{t}dxd\xi\\
=&\int_{{\rm G}(n,m)}\int_{\xi}
\left\|f|_\xi\right\|_1^{q-1}\left(\psi^*(\nabla\varphi)-\hat{\psi}_k^*(\nabla\varphi)\right) f\,d\HH^m\,d\xi\\
=&\int_{\Rn}\left(\psi^*-\hat{\psi}_k^*\right)\,d\widetilde{A}_{m,q}^{e}(f,\cdot).
\end{align*}
Therefore, \eqref{variation-phi+tchL-again} implies that
\begin{align*}
\liminf_{t\rightarrow 0^{+}}\frac{\widetilde{\Psi}^*_{m,q}(f_t)-\widetilde{\Psi}^*_{m,q}(f)}t=&\liminf_{t\rightarrow 0^{+}}\frac{\widetilde{\Psi}^*_{m,q}(f_t)-\widetilde{\Psi}^*_{m,q}(\hat{f}_{k,t})}t\\
&+\lim_{t\rightarrow 0^{+}}\frac{\widetilde{\Psi}^*_{m,q}(\hat{f}_{k,t})-\widetilde{\Psi}^*_{m,q}(f)}t\\
\geq &\int_{\Rn}\left(\psi^*-\hat{\psi}_k^*\right)\,d\widetilde{A}_{m,q}^{e}(f,\cdot)+\\
& +\int_{\Rn}\hat{\psi}_k^*\,d\widetilde{A}_{m,q}^{e}(f,\cdot)+\int_{S^{n-1}}h_{L_k}\,d\widetilde{A}_{m,q}^{s}(f,\cdot)\\
=&\int_{\Rn}\psi^*\,d\widetilde{A}_{m,q}^{e}(f,\cdot)+\int_{S^{n-1}}h_{L_k}\,d\widetilde{A}_{m,q}^{s}(f,\cdot),
\end{align*}
proving \eqref{logconv-g-core-lower}.

Since $\lim_{k\to\infty}L_k=L$, combining \eqref{logconv-g-core-upper} and \eqref{logconv-g-core-lower} yields Theorem~\ref{Amq-pair-measures}.
\end{proof}

\section{Optimization problem and the proof of Theorem~\ref{Amq-even-Minkowski-pair-measures}}
\label{secTheorem1.3}

This section first discusses the related extremal problem, and then solves the related even Minkowski-type problem. An important tool is
Lemma~3.4 in \cite{FR26}.

\begin{lemma}[Falah, Rotem \cite{FR26}]
\label{mu-nu-linear-lower-bound}
Let $\mu$ and $\nu$ be finite Borel measures on $\R^n$ and on $S^{n-1}$,  respectively, such that $\mu$ has finite first moment, no linear hyperplane contains the union of the supports of $\mu$ and $\nu$, and $\mu+\nu$ is centered in the sense that
\begin{equation}
\label{mu-nu-linear-lower-bound-cond}
\int_{\R^n}x\,d\mu(x)+\int_{S^{n-1}}u\,d\nu(u)=o.
\end{equation}
Then there exists $\alpha>0$ depending on $\mu$ and $\nu$ satisfying that for any $\varphi\in {\rm Conv}_c^n$ with $\min\varphi=\varphi(o)$, and for any $x\in\R^n$, we have
\begin{equation*}
\label{mu-nu-linear-lower-bound-eq}
\varphi(x)\geq \alpha\|x\|-\frac1{\mu(\R^n)}\left(\int_{\R^n}\varphi^*\,d\mu(x)+\int_{S^{n-1}}\bar{\varphi}^*\,d\nu(u)\right).
\end{equation*}
 \end{lemma}

We will also need Claim~3.6 from \cite{BLY26+}.

\begin{lemma}
\label{pancake}
Let $R>0$, $m\in\{1,\ldots,n-1\}$. There exist $\gamma,\varepsilon_0>0$ depending on $n,m,q,R$ such that if 
$F\subset R\,B^n$ is a convex body with $o\in{\rm int}\,F$, and $h_{F}(w), h_{F}(-w) \leq \varepsilon$ hold for some $w\in S^{n-1}$ and $\varepsilon\in(0,\varepsilon_0)$, then
\begin{align}
\label{pancake-eq}
\widetilde{\Psi}_{m,q}(F)\leq& \gamma \varepsilon^{\frac{\min\{q,m\}}2}&&\mbox{if }q>0,\\
\label{pancake-eq0}
%\widetilde{\Psi}_{m,q}(L)\leq& \gamma\log\varepsilon&&\mbox{if }q=0\\
\widetilde{\Psi}_{m,q}(F)\geq& \gamma\varepsilon^{\frac{q}2}&&\mbox{if }q<0,\\
\label{pancake-eq1}
\widetilde{\Psi}_{m,0}(F) \leq & \gamma \log \varepsilon &&\mbox{if }q=0.
\end{align}
\end{lemma}

We are ready to prove the key ingredient towards solving the even functional centro-sectional Minkowski problem, which are divided into two cases ($q\neq 0$ and $q=0$). 

\begin{prop}
\label{mu-nu-Psimq-optimization}
Let $q\in\R$, $m\in\{1,\ldots,n-1\}$.  let $\mu$ and $\nu$ be finite even Borel measures on $\R^n$ and on $S^{n-1}$,  respectively, such that $\mu(\R^n)>0$, $\mu$ has finite first moment, and no linear hyperplane contains the union of the supports of $\mu$ and $\nu$.
For 
\begin{align*}
\mathcal{F}_q:=&\left\{\varphi\in {\rm Conv}_c^n:\frac1q\log\widetilde{\Psi}_{m,q}(e^{-\varphi})\geq 0\mbox{ \ and }\varphi\mbox{ is even}\right\}&&\mbox{if }q\neq 0,\\
\mathcal{F}_0:=&\left\{\varphi\in {\rm Conv}_c^n:\widetilde{\Psi}_{m,0}(e^{-\varphi})\geq 0\mbox{ \ and }\varphi\mbox{ is even}\right\}&&\mbox{if }q= 0,\\
G(\varphi):=& \int_{\R^n}\varphi^*\,d\mu+\int_{S^{n-1}}\bar{\varphi}^*\,d\nu \mbox{ \ \ \ for }\varphi\in \mathcal{F}_q&&\mbox{and }q\in\R,
\end{align*}
there exists a  $\varphi_0\in \mathcal{F}_q$ such that
\begin{equation}
\label{mu-nu-Psimq-optimization-existence}
G(\varphi_0)=\min_{\varphi\in \mathcal{F}_q} G(\varphi).
\end{equation}
 \end{prop} 
\begin{proof}  First we prove the existence of the extremizer $\varphi_0\in \mathcal{F}_q$. For any $q\in\R$, there exists some $r_q>0$ such that 
\begin{equation*}
\label{mu-nu-Psimq-optimization-psirq}
\psi_{(q)}\in \mathcal{F}_q\mbox{ \ for \ }\psi_{(q)}(x)=\left\{
\begin{array}{lcl}
0 &\mbox{ if }&x\in r_qB^n,\\
\infty &\mbox{ if }&x\not \in r_qB^n.
\end{array} \right.
\end{equation*}
We choose a sequence $\varphi_{(k)}\in \mathcal{F}_q$, $k\geq 1$, such that
$$
\inf_{\varphi\in \mathcal{F}_q} G(\varphi)=\lim_{k\to\infty}G(\varphi_{(k)}), 
$$
and hence we may assume that $f_{(k)}:=
e^{-\varphi_{(k)}}$ satisfies 
\begin{equation}
\label{mu-nu-Psimq-optimization-Psimqfk1}
\widetilde{\Psi}_{m,q}(f_{(k)})=
\left\{\begin{array}{lcl}
1&\mbox{ if }&q\neq 0,\\
0&\mbox{ if }&q=0,
\end{array}\right.
\end{equation}
and 
$G(\varphi_{(k)})\leq G(\psi_{(q)})$ for any $k\geq 1$; or in other words,
$$
\int_{\R^n}\varphi_{(k)}^*\,d\mu+\int_{S^{n-1}}\bar{\varphi}_{(k)}^*\,d\nu\leq G(\psi_{(q)}).
$$
As $\mu$, $\nu$ and $\varphi_{(k)}$ are even, the conditions about centering  (cf. \eqref{mu-nu-linear-lower-bound-cond}) and about the minimality at $o$ automatically holds in Lemma~\ref{mu-nu-linear-lower-bound}, and in turn
 Lemma~\ref{mu-nu-linear-lower-bound} yields the existence of an $\alpha>0$ depending on $\mu$ and $\nu$ such that if $k\geq 1$ and $x\in\R^n$, then
\begin{equation}
\label{mu-nu-Psimq-optimization-phik-linear-lower}
\varphi_{(k)}(x)\geq \alpha\|x\|+\beta
\end{equation}
for $\beta=-\frac{G(\psi_{(q)})}{\mu(\R^n)}$. For $k\geq 1$, let
\begin{equation}
\label{mu-nu-Psimq-optimization-phik-Mk} 
N_{(k)}=\varphi_{(k)}(o)=\min \varphi_{(k)}\geq \beta. 
\end{equation}
We deduce from \eqref{mu-nu-Psimq-optimization-phik-linear-lower} and \eqref{mu-nu-Psimq-optimization-phik-Mk} that for any $k\geq 1$ and $x\in\R^n$, we have
\begin{equation}
\label{mu-nu-Psimq-optimization-fk-thetaNk}
f_{(k)}(x)=e^{-\varphi_{(k)}(x)}\leq\theta_{N_{(k)}}(x),
\end{equation}
where for $\ell \geq \beta$ and $R_{\ell}=\frac{\ell-\beta}{\alpha}\geq  0$, the log-concave function $\theta_\ell$ on $\R^n$ is defined by the formula
$$
\theta_\ell(x)=
\left\{\begin{array}{lcl}
e^{-\ell}&\mbox{ if }&\|x\|\leq R_\ell,\\[1ex]
\exp\left(-\alpha\|x\|-\beta\right)&\mbox{ if }&\|x\|\geq R_\ell.
\end{array}\right.
$$
We observe that if $\xi\in{\rm G}(n,m)$, then
$$
\int_\xi \theta_\ell\,d\HH^m=\omega_m e^{-\ell}R_{\ell}^m +m\omega_m
\int_{R_{\ell}}^\infty e^{-\alpha r-\beta}r^{m-1}\,dr:=I_\ell<\infty,
$$
and there exists $\ell_0>0$ such that $\frac{\ell}{2\alpha}\leq R_\ell\leq \frac{2\ell}{\alpha}$ if $\ell\geq \ell_0$,
and hence $\theta_\ell$ satisfies  
\begin{equation}
\label{mu-nu-Psimq-optimization-theta-ell}
\theta_{\ell}\geq \theta_{\ell'}\mbox{ \ if }\ell\leq \ell',\mbox{ \ \ and \ \ } \lim_{\ell\to\infty}I_\ell=0.
\end{equation}
We deduce from \eqref{mu-nu-Psimq-optimization-fk-thetaNk}, \eqref{mu-nu-Psimq-optimization-theta-ell} and from the conditions 
$\frac1q\log\widetilde{\Psi}_{m,q}(f_{(k)})\geq 0$ if $q\neq 0$ and $\widetilde{\Psi}_{m,0}(f_{(k)})\geq 0$  if $q=0$, that there exists an $N>0$ such that $N_{(k)}\leq N$ for every $k\geq 1$; therefore, if $k\geq 1$ and $x\in\R^n$, then
\begin{equation}
\label{mu-nu-Psimq-optimization-fk-phik-bounded}
f_{(k)}(x)\leq\theta_{\beta}(x)\mbox{ \ and } \beta\leq \inf \varphi_{(k)}=\varphi_{(k)}(o)\leq N.
\end{equation}
It follows from Lemma~\ref{epiconvergence-selection} (cf. \eqref{mu-nu-Psimq-optimization-phik-linear-lower}) that there exists a subsequence $\{\varphi_{(k')}\}$ and a $\varphi_0\in {\rm Conv}_c(\R^n)$ such that $\varphi_{(k')}\xrightarrow{\rm epi} \varphi_0$.

In order to show that $\varphi_0\in {\rm Conv}_c^n(\R^n)$ (and hence $o\in{\rm int}\,{\rm dom}\,\varphi_0$), we verify that there exists $r_0>0$ such that
\begin{equation}
\label{mu-nu-Psimq-optimization-r0-non-degenerate}
r_0B^n\subset \{\varphi_{(k')}\leq N+1\}
\end{equation}
for each $\varphi_{(k')}$ involved in the definition of $\varphi_0$ (cf. Definition~\ref{epiconvergence-def}). Here we use in an essential way that each  $\varphi_{(k')}$ is even.

Indirectly, we suppose that \eqref{mu-nu-Psimq-optimization-r0-non-degenerate} does not hold, and seek a contradiction. After possibly taking a further subsequence, we may assume that
\begin{equation*}
\label{mu-nu-Psimq-optimization-rk-degenerate}
\lim_{k'\to\infty}r_{k'}=0\mbox{ \ holds for }
r_{k'}=\max\left\{r\geq 0:rB^n\subset \{\varphi_{(k')}\leq N+1\}\right\}.
\end{equation*}
In particular, there exists $u_{k'}\in S^{n-1}$ for each $k'$ such that $x\cdot u_{k'}\leq r_{k'}$ holds for any $x\in \{\varphi_{(k')}\leq N+1\}$, and as $\varphi_{(k')}(0)\leq N$ (cf. \eqref{mu-nu-Psimq-optimization-fk-phik-bounded}), we have
\begin{equation*}
\label{mu-nu-Psimq-optimization-phik-uk-lower}
\varphi_{(k')}(s u_{k'})\geq N+\frac{s}{r_{k'}}\mbox{ \ holds if }s\geq r_{k'} 
\end{equation*}
by the convexity of $\varphi_{(k')}$.
For $M=e^{-\beta}$,  \eqref{mu-nu-Psimq-optimization-phik-Mk}, \eqref{mu-nu-Psimq-optimization-fk-thetaNk} and \eqref{mu-nu-Psimq-optimization-fk-phik-bounded} yield for any $k'$ that
\begin{align}
\label{mu-nu-Psimq-optimization-f0-NMo}
e^{-N}\leq f_{(k')}(o)\leq & M, \\
\label{mu-nu-Psimq-optimization-f0-NMx}
 f_{(k')}(x)\leq & M  &&\mbox{for }x\in\R^n. 
\end{align}

Next, we consider the corresponding body  $K_m(f_{(k')})$ (cf. \eqref{Ball-body}). It follows from \eqref{mu-nu-Psimq-optimization-fk-phik-bounded} that there exists $R_0>0$ independent of $k'$ such that for any $k'$, we have
\begin{equation}
\label{mu-nu-Psimq-optimization-Ball-body-R}
K_m(f_{(k')})\subset R_0B^n. 
\end{equation}
For the radial function $\varrho_{(k')}$ of $K_m(f_{(k')})$, we deduce that
\begin{align}
\nonumber
\varrho_{(k')}(u_{k'})^m\leq &me^N\int_0^{r_{k'}} s^{m-1}M\,ds+me^N\int_{r_{k'}}^\infty s^{m-1} \exp\left(-\left(N+\frac{s}{r_{k'}}\right)\right)\,ds\\
 \label{mu-nu-Psimq-optimization-Ball-body-rho}
\leq & e^NM\cdot r_{k'}^m +m\int_{0}^\infty s^{m-1} \exp\left(-\frac{s}{r_{k'}}\right)\,ds,
\end{align}
where the right hand side of \eqref{mu-nu-Psimq-optimization-Ball-body-rho} tends to zero as $k'$ tends to infinity because $r_{k'}$ tends to zero. Writing $h_{(k')}$ to denote the support function of $K_m(f_{(k')})$, we conclude the existence of some $w_{k'}\in S^{n-1}$ for each $k'$ such that $h_{(k')}(w_{k'})=h_{(k')}(-w_{k'})$ tends to zero as $k'$ tends to infinity; therefore, combining 
\eqref{mu-nu-Psimq-optimization-f0-NMo}, \eqref{mu-nu-Psimq-optimization-Ball-body-R}, Lemma~\ref{pancake} and Corollary~\ref{Psimqf-Ball-body} yields that
\begin{align*}
\lim_{k'\to\infty}\frac1q\log \widetilde{\Psi}_{m,q}(f_{(k')})=
\lim_{k'\to\infty}\frac1q\log \left(f_{(k')}(o)^q\widetilde{\Psi}_{m,q}\left(K_m(f_{(k')})\right)\right)=&-\infty,&&q\neq 0,\\
\lim_{k'\to\infty} \widetilde{\Psi}_{m,0}(f_{(k')})=
\lim_{k'\to\infty}\left(\log f_{(k')}(o)+\widetilde{\Psi}_{m,0}\left(K_m(f_{(k')})\right)\right)=&-\infty,&&q= 0.
\end{align*}
This fact contradicts that $\varphi_{(k')}\in\mathcal{F}_q$, and proves \eqref{mu-nu-Psimq-optimization-r0-non-degenerate}.

According to Definition~\ref{epiconvergence-def}, \eqref{mu-nu-Psimq-optimization-r0-non-degenerate} yields that
$r_0B^n\subset \{\varphi_0\leq N+1\}$, and hence $\varphi_0\in {\rm Conv}_c^n(\R^n)$. We deduce from
\eqref{mu-nu-Psimq-optimization-Psimqfk1} and Lemma~\ref{Psimqf-continuous2} that $f_0:=
e^{-\varphi_0}$ satisfies
$$
\widetilde{\Psi}_{m,q}(f_0)=
\left\{\begin{array}{lcl}
1&\mbox{ if }&q\neq 0,\\
0&\mbox{ if }&q=0;
\end{array}\right.
$$ 
therefore, we have completed the proof of \eqref{mu-nu-Psimq-optimization-existence}.
\end{proof}

\begin{proof}[Proof of Theorem~\ref{Amq-even-Minkowski-pair-measures}]
The necessary conditions for the functional centro-sectional measures are consequences of Proposition~\ref{tildeA-finiteness} and Lemma~\ref{Amq-pair-measures-support}.

To prove the suffiency part for the existence of functional centro-sectional measures in
 Theorem~\ref{Amq-even-Minkowski-pair-measures}, let $q\in\R$, $m\in\{1,\ldots,n-1\}$. Let $\mu$ and $\nu$ be finite even Borel measures on $\R^n$ and on $S^{n-1}$,  respectively, such that $\mu(\R^n)>0$, $\mu$ has finite first moment, and no linear hyperplane contains the support of both $\mu$ and $\nu$.
Using the notation of Proposition~\ref{mu-nu-Psimq-optimization}, we show that there exist a $\lambda>0$ such that for the extremal $\varphi_0\in\mathcal{F}_q$ provided by Proposition~\ref{mu-nu-Psimq-optimization}, we have
\begin{align}
\label{mu-nu-Psimq-optimization-mu}
\widetilde{A}_{m,q}^{e}\left(\lambda\,e^{-\varphi_0},\cdot\right)=&\mu,\\
\label{mu-nu-Psimq-optimization-nu}
\widetilde{A}_{m,q}^{s}\left(\lambda\,e^{-\varphi_0},\cdot\right)=&\nu.
\end{align}
In order to verify \eqref{mu-nu-Psimq-optimization-mu} and \eqref{mu-nu-Psimq-optimization-nu}, let $L\subset\Rn$  be an $o$-symmetric convex body, and let $\zeta \in C_c(\Rn)$ be an even function.
For either $\eta=h_L$ or $\eta=\zeta$, let $\bar{\eta}:\R^n\to[0,\infty)$ be defined by
$$
\bar{\eta}(x)=\lim_{r\to\infty}\frac{\eta(rx)}{r}=
\left\{\begin{array}{lcl}
h_L(x)&\mbox{ if }&\eta=h_L,\\
0&\mbox{ if }&\eta=\zeta.
\end{array}\right.
$$
We observe that there exist $\alpha_0,\beta_0>0$ such that $\eta(x)\leq \alpha_0\|x\|+\beta_0$ for any $x\in\R^n$, and hence
 $\bar{\eta}(u)\leq \alpha_0$ for any $u\in S^{n-1}$.
For suitable $t_0>0$ (depending on $q,\varphi_0,L,\zeta$) and $t\in (t_0,t_0)$, we consider
$$
\varphi_t=
\left\{\begin{array}{lcl}
(\varphi_0^*+t\eta)^*+\frac1q\,\log \widetilde{\Psi}_{m,q}(e^{-(\varphi^*_0+t\eta)^*})&\mbox{ if }&q\neq 0,\\[1ex]
(\varphi_0^*+t\eta)^*+ \widetilde{\Psi}_{m,0}(e^{-(\varphi^*_0+t\eta)^*})&\mbox{ if }&q=0,
\end{array}\right.
$$
and hence $\varphi_t\in {\rm Conv}_c^n(\R^n)$ is even; moreover, we deduce from \eqref{Psi-Amq-homogeneity}, \eqref{Legendre-shift}, \eqref{Legendre-monotone} and \eqref{Legendre-stars} that 
\begin{align}
%\label{mu-nu-Psimq-optimization-phit-Psimq}
\widetilde{\Psi}_{m,q}(e^{-\varphi_t})=&
\left\{\begin{array}{lcl}
1&\mbox{ if }&q\neq 0,\\
0&\mbox{ if }&q=0,
\end{array}\right.
 \notag\\
\label{mu-nu-Psimq-optimization-phit-eta}
\varphi_t^*\leq &
\left\{\begin{array}{lcl}
\varphi_0^*+t\eta-\frac1q\,\log \widetilde{\Psi}_{m,q}(e^{-(\varphi^*_0+t\eta)^*})&\mbox{ if }&q\neq 0,\\[1ex]
\varphi_0^*+t\eta- \widetilde{\Psi}_{m,0}(e^{-(\varphi^*_0+t\eta)^*})&\mbox{ if }&q=0,
\end{array}\right. \\
%\label{mu-nu-Psimq-optimization-phit-etabar}
\overline{\varphi_t^*}\leq &\overline{\varphi_0^*}+t\bar{\eta}.\notag
\end{align}
In particular, $\varphi_t\in\mathcal{F}_q$ whenever $t\in (-t_0,t_0)$. We consider the continuous function $\aleph_q(t)$ on $(-t_0,t_0)$, where, if $q\neq 0$,
\begin{align*}
\aleph_q(t)=&\int_{\mathbb{R}^n}\varphi_0^*+t\eta-\frac1q\,\log \widetilde{\Psi}_{m,q}(e^{-(\varphi^*_0+t\eta)^*})\,d\mu+\int_{S^{n-1}}\overline{\varphi_0^*}+t\bar{\eta}\,d\nu;\\
\aleph_0(t)=&\int_{\mathbb{R}^n}\varphi_0^*+t\eta- \widetilde{\Psi}_{m,0}(e^{-(\varphi^*_0+t\eta)^*})\,d\mu+\int_{S^{n-1}}\overline{\varphi_0^*}+t\bar{\eta}\,d\nu
\end{align*}
that is finite, since $\mu$ has finite first moment and both $\mu$ and $\nu$ are finite. Moreover, it is differentiable at $t=0$ (cf. Proposition~\ref{variation-phi+thL} and Proposition~\ref{variation-phi+tzeta}). We deduce from \eqref{mu-nu-Psimq-optimization-existence} and \eqref{mu-nu-Psimq-optimization-phit-eta} that if $t\in (-t_0,t_0)$, then
$$
\aleph_q(0)=G(\varphi_0)\leq G(\varphi_t)\leq \aleph_q(t);
$$
therefore, we conclude from $\widetilde{\Psi}_{m,q}(e^{-\varphi_0})=1$ for $q\neq 0$, Proposition~\ref{variation-phi+thL} and Proposition~\ref{variation-phi+tzeta}, and from applying \eqref{Psi-Amq-homogeneity} with $\lambda=\left(\mu(\R^n)\right)^{\frac1q}$ if $q\neq 0$  and $\lambda=1$ for $q=0$ that for any $q\in\R$, we have
\begin{align}
\nonumber
0=\aleph'_q(0)=&\int_{\mathbb{R}^n}\eta\,d\mu+\int_{S^{n-1}}\bar{\eta}\,d\nu\\
\label{mu-nu-Psimq-optimization-aleph-der}
&-\int_{\Rn}\eta\,d\widetilde{A}_{m,q}^{e}(\lambda f_0,\cdot)-\int_{S^{n-1}}\bar{\eta}\,d\widetilde{A}_{m,q}^{s}(\lambda f_0,\cdot),
\end{align}
where all integrals occurring in \eqref{mu-nu-Psimq-optimization-aleph-der} are finite. 
If $\eta=\zeta$ for a  $\zeta \in C_c(\Rn)$, then \eqref{mu-nu-Psimq-optimization-aleph-der} reads as
$$
0=\int_{\mathbb{R}^n}\zeta\,d\mu-\int_{\Rn}\zeta\,d\widetilde{A}_{m,q}^{e}(\lambda f_0,\cdot).
$$
As  $\zeta \in C_c(\Rn)$ is arbitrary, we deduce \eqref{mu-nu-Psimq-optimization-mu}. Next, if $\eta=h_L$ for an $o$-symmetric convex body $L\subset\R^n$ in \eqref{mu-nu-Psimq-optimization-aleph-der}, then we deduce from \eqref{mu-nu-Psimq-optimization-aleph-der} and $\mu=\widetilde{A}_{m,q}^{e}(\lambda f_0,\cdot)$ that
$$
0=\int_{S^{n-1}}h_L\,d\nu-\int_{S^{n-1}}h_L\,d\widetilde{A}_{m,q}^{s}(\lambda f_0,\cdot).
$$
In particular, if $L_1,L_2\subset\R^n$ are convex bodies with $o\in L_1,L_2$, then
$$
\int_{S^{n-1}}(h_{L_1}-h_{L_2})\,d\nu=\int_{S^{n-1}}(h_{L_1}-h_{L_2})\,d\widetilde{A}_{m,q}^{s}(\lambda f_0,\cdot).
$$
Since differences of the form $(h_{L_1}-h_{L_2})$ for such $L_1,L_2$ are dense in $C(S^{n-1})$, we conclude
  \eqref{mu-nu-Psimq-optimization-nu}, completing the proof of Theorem~\ref{Amq-even-Minkowski-pair-measures}.
\end{proof}

\noindent{\bf Acknowledgements.} The authors would like to thank Gaoyong Zhang for helpful discussions. 
K\'aroly J. B\"or\"oczky   is supported by the NKKP Advanced grant 150613, Jinrong Hu is supported by the Austrian Science Fund (FWF)
10.55776/ESP1358925, and Jiaqian Liu is supported by  the National Natural Science Foundation of China (12401252). 
The third author would also like to thank the Alfr\'ed R\'enyi Institute of Mathematics for its warm hospitality during her visit. Part of the research was performed at a workshop at the Erd\H{o}s Center.

\end{document}